\documentclass[reqno]{amsart}
\usepackage{amssymb, amsmath}
\usepackage{natbib}
\usepackage{lscape,comment}
\usepackage{color}
\usepackage{dsfont}
\usepackage{mathrsfs}
\usepackage{bm}

\usepackage[pdftex,plainpages=false,colorlinks,hyperindex,bookmarksopen,linkcolor=red,citecolor=blue,urlcolor=blue]{hyperref}

\DeclareMathAlphabet{\mathpzc}{OT1}{pzc}{m}{it}

\bibpunct{[}{]}{;}{n}{,}{,}

\newtheorem{te}{Theorem}[section]

\newtheorem{os}[te]{Remark}
\newtheorem{prop}[te]{Proposition}
\newtheorem{lem}[te]{Lemma}

\newtheorem{coro}[te]{Corollary}
\numberwithin{equation}{section}

\allowdisplaybreaks

\newcommand{\com}[1]{{\textcolor{blue}{#1}}}
\def \l { \left( }
\def \r {\right) }
\def \ll { \left\lbrace }
\def \rr { \right\rbrace }
\newcommand{\lb}{\left (}
\newcommand{\rb}{\right )}
\newcommand{\lbb}{\left [}
\newcommand{\rbb}{\right ]}
\newcommand{\labs}{\left |}
\newcommand{\rabs}{\right |}
\newcommand{\lbrb}[1]{\lb #1 \rb}
\newcommand{\lbbrbb}[1]{\lbb#1\rbb}

\newcommand{\lbbrb}[1]{\lbb#1\rb}

\newcommand{\lbcurly}{\left\{}
\newcommand{\rbcurly}{\right\}}

\newcommand{\abs}[1]{\labs#1\rabs}

\newcommand{\curly}[1]{\lbcurly#1\rbcurly}

\newcommand{\simi}{\stackrel{\infty}{\sim}}
\newcommand{\simo}{\stackrel{0}{\sim}}

\newcommand{\ind}[1]{1_{\ll #1 \rr}}

\newcommand{\limi}[1]{\lim_{#1\to \infty}}
\newcommand{\limsupi}[1]{\limsup_{#1\to \infty}}
\newcommand{\liminfi}[1]{\liminf_{#1\to \infty}}
\newcommand{\limo}[1]{\lim_{#1\to 0}}

\newcommand{\Es}{\mathds{E}}

\newcommand{\Rb}{\mathbb{R}}
\newcommand{\Pb}{P}

\newcommand{\Pbb}[1]{P\lb #1\rb}

\newcommand{\Eds}[1]{\Es\lbb #1\rbb}

\newcommand{\LL}{L\'{e}vy }

\begin{document}

	\title[]{Semi-Markov processes, integro-differential equations and anomalous diffusion-aggregation}
	\author[]{Mladen Savov}
	\author[]{Bruno Toaldo}
	\address[]{Institute of Mathematics and Informatics, Bulgarian Academy of Sciences, Akad.  Georgi Bonchev street
		Block 8 - 1113, Sofia (Bulgaria)}
	\address[]{Dipartimento di Matematica e Applicazioni ``Renato Caccioppoli'' - Universit\`{a} degli Studi di Napoli Federico II, Via Cintia, Monte S. Angelo - 80126, Napoli (Italy)}
\email[]{mladensavov@math.bas.bg}	
	\email[]{bruno.toaldo@unina.it}
	\date{\today}

		\begin{abstract}
In this article integro-differential Volterra equations whose convolution kernel depends on the vector variable are considered and a connection of these equations with a class of semi-Markov processes is established. The variable order $\alpha(x)$-fractional diffusion equation is a particular case of our analysis and it turns out that it is associated with a suitable (non-independent) time-change of the Brownian motion. The resulting process is semi-Markovian and its paths have intervals of constancy, as it happens for the delayed Brownian motion, suitable to model trapping effects induced by the medium. However in our scenario the interval of constancy may be position dependent and this means traps of space-varying depth as it happens in a disordered medium. The strength of the trapping is investigated by means of the asymptotic behaviour of the process: it is proved that, under some technical assumptions on $\alpha(x)$, traps make the process non-diffusive in the sense that it spends a negligible amount of time out of a neighborhood of the region $\text{argmin}(\alpha(x))$ to which it converges in probability under some more restrictive hypotheses on $\alpha(x)$.
\end{abstract}

	\maketitle

\tableofcontents

\section{Introduction}
In last years the interplay between anomalous diffusion phenomena and integro-differential (fractional type) equations have gained considerable attention by the scientific comunity. This is certainly due to the fact that fractional equations are very popular in applications and in the theoretical literature (see, for example, Meerschaert and Sikorskii \cite{FCbook} for general information). As non-local equations in the time-variable they are able to include memory effects in the evolution and this is certainly usefull in applications (see, for example, Hairer et al. \cite{hairer} for very recent developments, see \cite{FCbook} or Metzler and Klafter \cite{Metzler} for a review of classical applications or Georgiou and Scalas \cite{nicos}, Raberto et al. \cite{raberto}, Scalas \cite{scalas} for more exotic models). One of the first and more natural model is the so-called fractional diffusion related to the equation, for $\alpha \in (0,1)$,
\begin{align}
\partial_t^\alpha q = \frac{1}{2}\Delta q
\label{andiff} 
\end{align}
and with anomalous diffusion phenomena (see \cite{Metzler} for a review of these relationships). Equation \eqref{andiff} is said to be related with subdiffusive phenomena in the sense that the mean square displacement,
\begin{align}
(\Delta x)^2 \, = \, \int_{\mathbb{R}^d} | x-u|^2 q(t, x-u) du
\end{align}
behaves as $(\Delta x)^2 \sim Ct^\alpha$ as $t \to \infty$, for $\alpha \in (0,1)$, Hence any model which can be associated with \eqref{andiff} is less then diffusive, since $\alpha \in (0,1)$. From the probabilistic literature we know that the process associated with \eqref{andiff} is a Brownian motion time-changed with the inverse of an independent stable subordinator (this is due to Baeumer and Meerschaert \cite{fracCauchy}, see also Bazhlekova \cite{bazhlekova} for pioneering results on the fractional Cauchy problem). The resulting process is also called delayed Brownian motion (Magdziarz and Schilling \cite{magda}) since its sample paths remain constant for time-intervals determined by the jumps of the stable subordinator: hence the process is delayed in the sense that the Brownian paths are stretched by the random time-change (see \cite{magda} also for a detailed investigation of the asymptotic of the delayed Brownian motion, or Capitanelli and D'Ovidio \cite{mirkorafasym} for asymptotic properties of diffusion time-changed via independent inverse subordinators). Taken in full generality, the equation \eqref{andiff} has the form
\begin{align}
\frac{d}{dt} \int_0^t (q(s,x)-q(0,x)) \, k(t-s) \, ds \, = \, G q(t, x)
\label{genintro}
\end{align}
where $G$ generates a Markov process $M$. The corresponding process is the time-change of $M$ with the inverse of an independent subordinator whose \LL measure $\nu(\cdot)$ is given by $\nu(t, \infty)=k(t)$ (this is due to Chen \cite{zqc}). We can say that this kind of processes are delayed in the same sense as for the delayed Brownian motion, since the paths remain constants due to the jumps of the corresponding subordinator.
Intervals of constancy are a classical feature of semi-Markov processes (see Harlamov \cite{harlamov} for the modern formulation of the corresponding theory) and it is true indeed that the delayed Brownian motion and in general delayed Markov processes are semi-Markov (see Cinlar \cite{cinlarsemi} and also Meerschaert and Straka \cite{meerstra} for the interpretation as limit of continuous time random walks). Hence the memory described by non-locality of the equation (in the time variable) is introduced in the sense that the lack of memory of the exponential distribution is lost, due the interval of constancy.

The intervals of constancy of the delayed Brownian motion are suitable to model trapping effects induced by the medium, in case the traps are homogeneous in space. However the traps are often of space-varying depth in the sense that the strength of the trapping may be position dependent as it happens in a disordered medium (e.g. \cite{fedofalco, barkai, stick, strakafedo, wong}). In the present paper we provide a model which is suitable to include this heterogeneity. The starting point is the ``variable order'' generalization of \eqref{genintro}, i.e., the equation
\begin{align}
\frac{d}{dt} \int_0^t (q(s,x)-q(0,x)) \, k(t-s,x) \, ds \, = \, G q(t, x).
\label{volterra}
\end{align}
It is clear that \eqref{volterra} specializes to
\begin{align}
\partial_t^{\alpha(x)} q \, = \,\frac{1}{2} \partial_x^2 q
\label{andifva}
\end{align}
by suitably choosing $k$. First we provide a connection of \eqref{volterra} with semi-Markov processes. It turns out that to construct the corresponding process one has to consider a Markov additive process $\l M_t, \sigma_t \r$ where the additive component $\sigma_t$ is strictly increasing and has a time-dependent \LL measure which is determined by the path (the current position) of the Markov process $M_t$. This construction will be made precise in Section \ref{sec2} by means of the theory of Markov additive processes introduced by Cinlar \cite{cinlar1, cinlar2}. In the case of \eqref{andifva} the first coordinate $M_t$ is given by a Brownian motion. Now let $L_t= \inf \ll s \geq 0: \sigma(s)>t \rr$ and define $X(t):=M(L(t))$. It is clear that $M$ and $L$ are now dependent processes and thus the random length of the intervals of constancy which are determined by the jumps of $\sigma$ depends on the position of $M$. This gives rise to a very heterogeneous behaviour of the process: the trapping effect induced by the time-change is space-varying. When $\sigma_t$ behaves locally as an $\alpha(x)$-stable subordinator whose order $\alpha(x)$ is determined by the position of $M_t$ the process $M(L(t))$ is associated with \eqref{andifva}. We found under some technical assumptions on $\alpha(x)$ that, a.s.,
\begin{align}
t^{-1} \int_0^t \ind{ M(L(w)) \in A} dw \sim 1 \text{ as } t \to \infty,
\label{asocc}
\end{align}
where $A$ is a suitable neighborhood of the region $\text{argmin}(\alpha(x))$ and the condition on $\alpha(x)$ depends on the structure of $A$. For example suppose that there exists $\beta>0$ small enough such that the region $A_\beta:= \ll x \in \mathbb{R}: \alpha(x)<\alpha_{\min}+\beta <1 \rr$ is bounded and has a Lebesgue null boundary, then \eqref{asocc} holds true for $A_\beta$, for all $\beta \leq \beta_0$ and some $\beta_0>0$, if $2\alpha_{\min}< \min \l \lim_{x \to \infty} \alpha(x) , \lim_{x \to -\infty} \alpha(x) \r$ but if instead one has that the $\alpha(x)$ satisfies $2\alpha_{\min}> \min \l \lim_{x \to \infty} \alpha(x) , \lim_{x \to -\infty} \alpha(x) \r$ there is attraction to infinity and so the process is diffusive, i.e., for all $\beta \leq \beta_0$ and all $K>0$, a.s.,
\begin{align}
t^{-1} \int_0^t \ind{M(L(w)) \in A^c_\beta \cap [-K,K]^c}dw \sim 1 \text{ as } t \to \infty. 
\end{align}
Another case we cover is when $\text{argmin}(\alpha(x))$ is unbounded and $\text{argmin}(\alpha(x)) = A_\beta$ for all $\beta$ small enough: here the conditions on $\alpha(x)$ can be relaxed depending on $\lim_{x \to \infty} \mathpzc{l} \l \text{argmin} (\alpha(x)) \cap [-x,x] \r$, where $\mathpzc{l}(\cdot)$ is the Lebegue measure. In Section \ref{secconv} we cover several situations of this type. 
Hence the trapping effect induced by the time-change, depending on $\alpha(x)$, can be so much stronger in the region $\text{argmin}(\alpha(x))$ than in the rest of $\mathbb{R}$ that the amount of time spent by the process in that position grows linearly with $t$, as $t \to \infty$, a.s.
Further when $\alpha (x)$ satisfies some more restrictive conditions (including that the set $\text{argmin} (\alpha(x))$ is a union of intervals and $x \mapsto \alpha(x)$ jumps on the minimum) we have proved that
\begin{align}
\lim_{t \to \infty }P^x \l M(L(t)) \in \text{argmin} (\alpha(x) )\r =1.
\end{align}
Hence the behaviour of the resulting process is so heterogeneous that it is completely far from a diffusion since it spends a negligible amount of time far from $\text{argmin}(\alpha(x))$ and in some cases the whole probability mass converges to the region $\text{argmin}(\alpha(x))$.

We call this phenomenon anomalous aggregation, inspired by Fedotov \cite{fedopre} who observed such a behaviour in the context of chemotaxis and anomalous subdiffusive transport. We remark that aggregating phenomena in the context of anomalous diffusion have been observed in other situations (e.g. \cite{campos, barkai, shushin}), and that a connection with fractional order equations has been argued in \cite{checgore, fedofalco}).

 \section{Construction of the process}
 \label{sec2}
We recall in this section some facts from the theory of Markov additive processes and semi-Markov processes (for this we refer to Cinlar \cite{cinlar2, cinlarlevy, cinlarsemi}) and we introduce our assumptions from the point of view of this theory.
 \subsection{Additive processes}
 \label{map}
Let $\l \Omega, \mathcal{F}, \mathcal{F}_t, M_t, \theta_t, P^x \r$ be a Markov process on $\mathbb{R}^d$ and let $\sigma= \ll \sigma_t ; t \geq 0 \rr$, be a family of functions from $\l \Omega, \mathcal{F} \r$ into $\l \mathbb{R}^m, \mathcal{B} \l \mathbb{R}^m \r \r$. Then $(M, \sigma) = \l \Omega, \mathcal{F}, \mathcal{F}_t, M_t, \sigma_t, \theta_t, P^x \r$  is said to be a Markov additive process if it holds that \cite[Definition 1.2]{cinlar2}
\begin{enumerate}
\item $t \mapsto \sigma_t$ is right-continuous with left limits, $\sigma_0=0$, $\sigma_t=\sigma_\zeta$ for any $t \geq \zeta$,
\item for each $t \geq 0$, $\sigma_t :  \Omega  \mapsto  \mathbb{R}^m$, is measurable with respect to $\mathcal{F}_t$ and $\mathcal{B} \l \mathbb{R}^m \r$,
\item for each $t \geq 0$, $B \in \mathcal{B}\l \mathbb{R}^d \r$, $B^\prime \in \mathcal{B} \l \mathbb{R}^m \r$ the mapping $\mathbb{R}^d \ni x \mapsto P^x \l M_t \in B, \sigma_t \in B^\prime \r \in [0,1]$ is in $\mathcal{B} \l \mathbb{R}^d \r$,
\item \label{theta} for each $s,t \geq 0$, $\sigma_{t+s} = \sigma_t + \sigma_s \circ \theta_t$, a.s.,
\item for each $s,t \geq 0$, $x \in \mathbb{R}^d$, $B \in \mathcal{B}\l \mathbb{R}^d \r$, $B^\prime \in \mathcal{B} \l \mathbb{R}^m \r$, 
\begin{align}
P^x \l M_s \circ \theta_t \in B, \sigma_s \circ \theta_t \in B^\prime \mid \mathcal{F}_t \r \, = \, P^{M(t)}  \l M_s \in B, \sigma_s \in B^\prime \r.
\end{align}
\end{enumerate} 
In common situations, and in this paper, the second coordinate $\sigma$ is one-dimensional and striclty increasing. Note that conditionally on a path $M(s), s \leq t$ the process $\sigma(t)$ has independent increments and it can be decomposed analogously to the L\'evy's decomposition as
\begin{align}
\sigma = A + \sigma^f + \sigma^c + \sigma^d
\end{align}
where $A$ is an additive functional of $M$ (a drift component), $\sigma^f$ is a purely discontinuous process whose jump are fixed by $M$, $\sigma^c$ is a continuous component and $\sigma^d$ is stochastically continuous. If one assume that $M(y)$ is a Hunt process with a reference measure and that $\sigma = \sigma^d$ is strictly increasing one can apply \cite[formula (2.23)]{cinlarlevy} to say that
\begin{align}
\mathds{E}\left[ e^{-\lambda \sigma(y) } \mid M(s), s \leq y \right] \, = \, e^{-\int_0^\infty \l 1-e^{-\lambda s} \r \int_0^t \nu(ds, M_y) dH_y}
\label{levyrepr}
\end{align}
where $H_y$ is a continuous  additive functional of $M$ and $\nu(\cdot, x)$ is a family of L\'evy measures supported on $(0,\infty)$ parametrized by $x$. In the present paper we deal with this kind of processes to construct a class of semi-Markov processes governed by \eqref{volterra} (see also \cite{kaspi} for construction of semi-Markov processes with regenerative sets).

 \subsection{The semi-Markov model}
 \label{defass}
In the present paper the process $M$ defined as $M=\l \Omega, \mathcal{F}, \mathcal{F}_y,  M_y , \theta_y, P^x \r$ will be a Hunt process on $\l \mathbb{R}, \mathcal{B} \l \mathbb{R}^d \r \r$, $\mathcal{F}_y$, i.e., it is right-continuous, $y \mapsto M_y$ is a.s. right-continuous, $M$ is normal and strong Markov with respect to $\mathcal{F}_y$ and quasi-left-continuous on $[0, \infty)$ (the process is non-explosive). It will be further true that $M_y$ is a Feller process, and thus it is associated with a semigroup of operator $\ll T_y \rr_{y \geq 0}$ defined by $\l T_yu \r(x) = \mathds{E}^x u(M_y)$, such that $T_y:C_0 \l \mathbb{R}^d \r \mapsto C_0 \l \mathbb{R}^d \r$, where $C_0 \l \mathbb{R}^d \r$ denotes the space of continuous functions on $\mathbb{R}^d$ vanishing at infinity, and strongly continuous in the sup-norm $\left\| \cdot \right\|$, i.e. $\left\| T_yu-u \right\| \to 0$ as $y \to 0$. The process $\l \Omega, \mathcal{F}, \mathcal{F}_y,  M_y , \sigma_y ,\theta_y, P^x \r$ will be an additive process with $\sigma_y$ one-dimensional, strictly increasing and constructed as follows. Let $D \in \mathbb{R}^+ \times \mathbb{R}^d$ be a Borel set and define
\begin{align}
\mu_M(D) \, = \, \mathpzc{l} \l  \ll  y \geq 0 : (y, M(y)) \in D \rr \r.
\end{align}
where $\mathpzc{l}$ is the Lebesgue measure. For Borel sets $D= A \times S$ the measure $\mu$ gives, informally, the amount of time in $A$ spent by $M_y$ in $S \in \mathcal{B} \l \mathbb{R}^d \r$. When $A$ is fixed we may define the measure on $\l \mathbb{R}^d, \mathcal{B} \l \mathbb{R}^d \r\r$ 
\begin{align}
\mu_{M,A}(S) \, : = \, \mu_M(A \times S).
\end{align}
By the definition of occupation measure we have that the identity
\begin{align}
\int_A u(M(y)) \mathpzc{l}(dy) \, = \, \int_{\mathbb{R}^d} u(x) \mu_{M,A}(dx)
\end{align}
is valid for every (measurable) non-negative function $u$ on $\mathbb{R}^d$. Hence we may assume on the line of \eqref{levyrepr} that by fixing $A= [0,y]$ we have
\begin{align}
\mathds{E}^x \left[ e^{-\lambda \sigma(y)}  \right] \, = \, \mathds{E}^xe^{- \int_0^\infty \l 1-e^{-\lambda s} \r \int_{\mathbb{R}^d} \nu(ds, w) d\mu_{M, [0,y]}(dw)},
\label{laplexpdef}
\end{align}
where $\nu(\cdot,w)$ is, for any $w \in \mathbb{R}^d$, the L\'evy measure of some subordinator, i.e., it is supported on $(0, \infty)$ and such that the integrability condition
\begin{align}
\int_0^\infty (s \wedge 1) \nu(ds, w) \, < \infty
\end{align} 
is fulfilled  for any $w \in \mathbb{R}^d$.
Hence if $\mu_{M,A}(dw)$ is absolutely continuous with respect to the Lebesgue measure one has that
\begin{align}\label{eq:sigma}
\mathds{E}^x \left[ e^{-\lambda \sigma(y)} \mid M(w), w \leq y \right] \, = \, e^{ -\int_0^\infty \l 1-e^{-\lambda s} \r \int_{\mathbb{R}^d} \nu(ds, w) l_{M,[0,y]}(w)dw}
\end{align}
where $l_{M, [0,y]}(w)$ is the Radon-Nycodim derivative (local time of $M$ at $w$). Of course one can choose a version of the local time such that $l_{X,[0,y]}(w, \omega)$ is a well defined r.v. for every $\omega$ so $l_{X,[0,y]}(w, \omega)$  is measurable $( \mathbb{R}^+ \times \Omega \mapsto \mathbb{R}^d ) $; in the end
\begin{align}
\mathds{E}^x \left[ e^{-\lambda \sigma(y)} \mid M(w), w \leq y \right] \, = \, e^{- \int_0^\infty \l 1-e^{-\lambda s} \r \int_{\mathbb{R}^d} \nu(ds, w) d\mu_{M, [0,y]}(dw, \omega)}.
\end{align}
We will use the notation
\begin{align}
\mathds{E}^x \left[ e^{-\lambda \sigma(y)} \mid M(w), w \leq y \right] \, = \, e^{-\int_0^y f \l \lambda, M_w \r dw}
\end{align}
where the functions
\begin{align}
[0, \infty) \times \mathbb{R}^d \mapsto f(\lambda, x) \, = \, \int_0^\infty \l 1-e^{-\lambda s} \r \nu(ds, x)
\end{align}
are such that $\lambda \mapsto f(\lambda, x)$ are a family of Bernstein functions parametrized by $x \in \mathbb{R}^d$. We remark that $f(\lambda, x)$ can be viewed as the Laplace exponents of the subordinators representing the increments of $\sigma$ when $M_w=x$ (see Schilling et al. \cite{librobern} for further information on Bernstein functions).
We remark that the process $\l M_t, \sigma_t \r$ is a strong Markov process adapted to $\mathcal{F}_t$, and the strong Markov property holds in the sense that, for any $\mathcal{F}$ random variable $Z$ and $\mathcal{F}_t$ stopping time $T$, one has
\begin{align}
\mathds{E}^x \left[ Z \circ \theta_T \mid \mathcal{F}_T \right] 	\, = \, \mathds{E}^{M(T)} [Z].
\label{stmp}
\end{align}
We remark that the process $\l M, \sigma \r = \l \Omega, \mathcal{F}, \mathcal{F}_t, M_t, \sigma_t, \theta_t, P^x \r $ is not a Markov process in the classical sense only because of the action of $\theta_t$ (Item \ref{theta} of Section \ref{map}).

Consider now the process $M_y$ at the time $t= \sigma(y)$. One can easily let the second coordinate of the Markov additive process $\l M_y, \sigma_y \r$ take value on the whole real line by considering the couple process $\l M_y, z+\sigma_y \r$ for $z \in \mathbb{R}$, and we define for $t \in \mathbb{R}$,
\begin{align}
L(t) = \inf \ll s \geq 0 : z+\sigma(s) > t \rr.
 \end{align} 
 Then consider the random set
\begin{align}
\mathcal{R}: = \overline{\ll z+ \sigma(y): 0 \leq y < \infty \rr}
\end{align}
which is the range of $z+\sigma(y)$ so that
\begin{align}
\mathcal{R} = \ll z+ \sigma(y): 0 \leq \sigma(y) < \infty \rr \cup \ll z+\sigma(y-) : y \in \mathcal{J} \rr
\end{align}
where
\begin{align}
\mathcal{J} = \ll 0 \leq s < \infty : \sigma(s)-\sigma(s-)>0 \rr,
\end{align}
and so
\begin{align}
\mathcal{R}^c = \bigcup_{s \in \mathcal{J}} (z+\sigma (s-), z+\sigma(s)).
\end{align}
Then let
\begin{align}
&g(t) = \sup \ll s<t: s \in \mathcal{R} \rr
&H(t) = \inf \ll s > t: s \in \mathcal{R} \rr.
\label{gh}
\end{align}
If we denote $\sigma^z(y) : = z+\sigma(y)$ we can rewrite the quantities \eqref{gh} as
\begin{align}
g(t) = \sigma^z (L(t)-), \qquad H(t) = \sigma^z (L(t)).
\end{align}
Finally we are ready to define for any $t \in \mathbb{R}$
\begin{align}
 X(t) = M(L(t)), \qquad    g(t) \leq t < H(t). 
\label{defx}
\end{align}
Note that definition \eqref{defx} is valid also for $t<0$ and is equivalent to
\begin{align}
X(t) = M(y), \qquad  \sigma^z(y-) \leq t < \sigma^z(y).
\end{align}
Note that the process $X(t)$ is a semi-Markov process in the sense that it enjoys the Markov property at any stopping time $T$ such that
\begin{align}
T(\omega) \in \ll s : \sigma_y(\omega) = s \text{ for some y} \rr,
\end{align}
see the discussion in \cite[Section 4b]{cinlarsemi}.
The semi-Markov property can be equivalently viewed in the sense of Gihman and Skorohod (see Gihman and Skorohod \cite[III.3]{gihman} or Harlamov \cite[III.12]{harlamov}): define
\begin{align}
\gamma(t) \, : = \, t - 0 \vee \sup \ll s \leq t : X(s) \neq X(t) \rr, \qquad t \geq 0,
\end{align}
then one has that the couple process $\l X_t, \gamma_t \r$ is a (strong) Markov process (compare with Meerschaert and Straka \cite[Section 4]{meerstra}).

\section{The governing integro-differential equation}
Let $\Pi_t$ be the operator
\begin{align}
(\Pi_tu)(x) \, := \, \mathds{E}\left[ u(X(t)) \mid X(0) = x \right].
\end{align}
In this section we establish a connection between the mapping
\begin{align}
t \mapsto \Pi_t u, 
\end{align}
for suitable functions $u$, and
the equation
\begin{align}
\frac{d}{dt} \int_0^t \l q(s, \cdot) -q(0, \cdot)\r \, \bar{\nu}(t-s,\cdot) ds \, = Gq(t)
\label{eqdiscussion}
\end{align}
where
\begin{align}
\l \mathfrak{D}_t^\cdot q(t) \r (\cdot) := \frac{d}{dt} \int_0^t \l q(s, \cdot) -q(0, \cdot)\r \, \bar{\nu}(t-s,\cdot) ds \
\label{defdfrac}
\end{align}
and, for any $s>0$, $x \in \mathbb{R}^d$,
\begin{align}
\bar{\nu}(s, x) \,  : = \, \nu((s, \infty), x)
\end{align}
and $G$ is the generator of the Markov process $M$. In what follows we will often write as above $q(t)$ instead of $q(t,\cdot)$ or $q(t,x)$, when the dependence on the vector variable $x \in \mathbb{R}^d$ is not used.
We will show that $t \mapsto \Pi_tu$ satisfies \eqref{eqdiscussion} in the mild sense, see below for the definition of mild solution.
Let us remark that in the case
\begin{align}
\nu(ds, x) \, = \, \frac{\alpha(x) s^{-\alpha(x)-1}}{\Gamma(1-\alpha(x))}ds,
\end{align}
for $\alpha(x)$ strictly between zero and one, then one has
\begin{align}
\bar{\nu}(s, x) \, = \, \frac{s^{-\alpha(x)}}{\Gamma(1-\alpha(x))}.
\label{tailfrac}
\end{align}
By substituting \eqref{tailfrac} in \eqref{defdfrac} we obtain the fractional derivative of variable order $\alpha(x)$: this is because when $x \mapsto \alpha(x)$ is constant the operator becomes a genuine fractional derivative of order $\alpha \in (0,1)$ called the regularized fractional Riemann-Liouville derivative and also Dzerbayshan-Caputo derivative (see Meerschaert and Sikorskii \cite[Chapter 2]{FCbook} for a complete discussion). When $x \mapsto \alpha(x)$ is constant the genuine time-fractional equation has a well-known probabilistic interpretation since Baeumer and Meerschaert \cite{fracCauchy}: take $\sigma^\alpha$ an $\alpha$-stable subordinator independent from the Markov process $M$, let $L^\alpha(t):= \inf \ll s \geq 0 : \sigma^\alpha(s) >t \rr$ and define $X(t)=M\l L^\alpha(t) \r$, then the mean value $\mathds{E}^x u(X(t))$ satisfies the time-fractional equation. When the subordinator considered is not necessarily stable, but a general subordinator $\sigma^f$ with Laplace exponent $f$ independent from $M$ then the equation governing the mean value of $X(t)=M \l L^f(t) \r$ has been written down in different forms by several authors (e.g. \cite{zqc, kochu, KoloCTRW, kololast, magda, meertri, meerpoisson, meertoa, toaldopota, toaldodo}). The more general and at the same time explicit approach is proposed by Chen \cite{zqc}: if $M$ is a Markov process associated with a semigroup on some Banach space $\mathfrak{B}$ generated by $G$ and $\sigma^f (t)$ is an independent strictly increasing subordinator with Laplace exponent $f(\lambda)$ and inverse process $L^f(t)$ then $q(x,t):=\mathds{E}^x u(M(L^f(t)))$ is the unique solution to
\begin{align}
\frac{d}{dt} \int_0^t \l q(s) -q(0)\r \, \bar{\nu}(t-s) ds \, = \, G q(t), \qquad q(0) =u \in \text{Dom}(G).
\label{nodepx}
\end{align}
The reader can consult Capitanelli and D'Ovidio \cite{mirkoraf} for a different approach based on Dirichlet forms, Meerschaert et al. \cite{meerbounded} for a detailed study of time-fractional equations on bounded domains, also Bazhlekova \cite{bazh15} for an analytical study of integro-differential equations of the form \eqref{nodepx} with completely monotone kernels and  Beghin and Ricciuti \cite{costa} or Orsingher et al. \cite{orsrictoapota} for variable order $\alpha(t)$ equations.
Our equation \eqref{eqdiscussion} is more general in the sense that the kernel of the convolution in \eqref{defdfrac} depends on the vector variable $x$. Equations having this form have been considered, from a probabilistic point of view, in Baeumer and Straka \cite[Section 6.2]{baemstra} (the fractional case) and the associated processes are obtained as limit of continuous times random walks, and Orsingher et al. \cite{orsrictoasemi}. In this paper the authors considered a Markov additive process $\l M_t, \sigma_t \r$ where $M_t$ is a Markov chain and $\sigma_t$ depends on $M_t$ as in \eqref{laplexpdef} and proved that the mean value $\mathds{E}^x u(M(L(t)))$ satisfies \eqref{eqdiscussion} with 
\begin{align}
(Gu)(x) \, = \, \theta_x \int \l u(y)-u(x) \r h_x(dy)
\end{align}
where $h_x(\cdot)$ are the transition probabilities of the jump chain embedded in $M$ and $\theta_x$ the parameters of the exponential waiting times. See also Garra et al. \cite{garra} for the variable order fractional equation governing a counting process.

In this section we show that the equation \eqref{eqdiscussion} governs in the mild sense the mean value of general semi-Markov processes, obtained as a time-change, when $M_t$ is not necessarily stepped.
We will assume throughout this section that the processes $M_t$ and $\l M_t, \sigma_t \r$ are Feller processes and thus they are associated with semigroups of operators, respectively, $T_t$ and $P_t$, which map the Banach space of continuous functions vanishing at infinity (on $\mathbb{R}^d$ and $\mathbb{R}^d \times [0, \infty)$) equipped with the sup-norm $\left\| \cdot \right\|$, into itself. The semigroups are also strongly continuous, i.e., they are such that $\left\| T_tu-u \right\| \to 0$ as well as $\left\| P_th-h \right\| \to 0$ for $t \to 0$ for any $u \in C_0 \l \mathbb{R}^d \r$ and $h \in C_0 \l \mathbb{R}^d \times [0, \infty) \r$. We will denote the generators of $T_t$ and $P_t$, respectively, $\l G, \text{Dom}(G) \r$ and $\l A, \text{Dom}(A) \r$. Recall that the generator is the operator
\begin{align}
Gu:= \lim_{t \to 0} \frac{T_tu-u}{t}
\end{align}
with domain
\begin{align}
\text{Dom}(G) \, := \, \ll u \in C_0 \l \mathbb{R}^d \r: \lim_{t \to 0} \frac{T_tu-u}{t} \text{ exists as uniform limit} \rr.
\end{align}
We will assume that $C_c^\infty \l \mathbb{R}^d \r \subset \text{Dom}(G)$ and we know that this implies (e.g. \cite[Theorem 2.21]{schillinglevy}) that $G$ has the form
\begin{align}
(Gu) (x) \, = \, &-c(x) u(x)+l(x) \cdot \nabla u(x) + \frac{1}{2} \text{div} Q(x) \nabla u(x) \notag \\
& + \int_{\mathbb{R}^d} \l u(x+y)-u(x)- \nabla u (x) \cdot y  \chi (|y|)  \r N(x,dy)
\label{reprg}
\end{align}
where $c(x) \geq 0$, $(l(x), Q(x), N(x, \cdot))$ is a L\'evy triplet for any fixed $x \in \mathbb{R}^d$ with $Q(x) \in \mathbb{R}^{d\times d}$ symmetric a positive semidefinite and $N(x, \cdot)$ satisfies
\begin{align}
\int_{\mathbb{R}^d- \ll 0 \rr} \l |y|^2 \wedge 1 \r N(x, dy) < \infty,
\end{align}
while the non-negative bounded function $\chi$ is  a truncation function such that $0 \leq \chi (s) \leq 1-\kappa \l s \wedge 1 \r$ for some $\kappa >0$ and $s \chi (s)$ remains bounded. It is well known further that under these assumptions the operator $G$ has the form \cite[Corollary 2.23]{schillinglevy}
\begin{align}
Gu(x)  \, = \, -q(x, D) u(x) \, := \, - \int_{\mathbb{R}^d}  e^{i x \cdot \xi} q(x, \xi) \widehat{u} (\xi) d\xi
\label{symbg}
\end{align}
where $q(x, \xi)$ is a continuous negative definite function with representation
\begin{align}
q(x, \xi) \, = \,& q(x,0) -il(x) \cdot 	\xi + \frac{1}{2} \xi \cdot Q(x) \xi  \\& +\int_{\mathbb{R}^d - \ll 0 \rr} \l 1-e^{iy \cdot \xi} + i\xi \cdot y \chi (|y|)  \r N(x, dy),
\label{qx}
\end{align}
and 
\begin{align}
\widehat{u}(\xi) \, := \, (2\pi)^{-d} \int_{\mathbb{R}^d} e^{-ix \cdot \xi} u(x) dx.
\end{align}
We will further assume that $q(x,0) = 0$ and that $q$ has bounded coefficients in the sense of \cite[eq. (2.33)]{schillinglevy}, i.e., 
\begin{align}
\sup_{x \in \mathbb{R}^d} |q(x,0)|+ \sup_{x \in \mathbb{R}^d} |l(x)| + \sup_{x \in \mathbb{R}^d}|Q(x)| + \sup_{x \in \mathbb{R}^d} \int_{\mathbb{R}^d} \l |y|^2 \wedge 1 \r N(x,dy) < \infty
\label{boundcoff}
\end{align}
and hence we know (e.g \cite[Theorem 2.33]{schillinglevy}) that $T_t$ is conservative and $x \mapsto q(x, \xi)$ is a continuous function.

With the forthcoming result we characterize the generator of the couple process $\l M_t, \sigma_t \r$.
\begin{prop}
\label{propgen}
Assume that the strong Markov processes $M_t$ and $\l M_t, \sigma_t \r$ are Feller processes associated with the semigroups of operators $T_t$ and $P_t$ as above. Let $A$ be the generator of $P_t$ and assume that $C_c^\infty \l \mathbb{R}^d \r \subset \text{Dom}(G)$ as well as $C_c^\infty (\mathbb{R}^d \times [0, \infty)) \subset \text{Dom}(A)$ and so $G$ and $A$ are pseudo-differential operators; let $q(x, D)$ defined as in \eqref{symbg} be the symbol of $G$. Assume $q(x,0) = 0$ in \eqref{qx} and that $q$ has bounded coefficients in the sense of \eqref{boundcoff}. Further let $x \mapsto f(\lambda, x)$ be continuous and such that
\begin{align}
\sup_{x \in \mathbb{R}^d} \int_0^\infty (s \wedge 1) \nu(x, ds) < \infty.
\label{boundlev}
\end{align}
Then we have that the Feller process $(M_t, \sigma_t)$ is generated by $\l A, \text{Dom}(A) \r$ where $A$ has the form
\begin{align}
(Ah)(x,z) \, = \, & (G h)(x,z) \, + \, \int_0^\infty \l h(x, z+w)-h(x,z) \r \nu(x, dw).
\label{reprgen}
\end{align}
\end{prop}
\begin{proof}
Observe that, for $h \in C_c^\infty \l \mathbb{R}^d \times [0, \infty) \r$, we have
\begin{align}
&\int_{\mathbb{R}^{d+1}} e^{i \xi_1 x+i\xi_2 z}\l q(x, \xi_1) - f(-i\xi_2,x) \r \widehat{h}(\xi_1, \xi_2 ) d\xi_1 d\xi_2\notag \\
 = \, &\int_{\mathbb{R}^{d+1}} e^{i \xi_1 x+i\xi_2z} q(x, \xi_1) \widehat{h}(\xi_1, \xi_2)\, d\xi_1 d\xi_2 \notag \\ &- \int_{\mathbb{R}^{d+1}} e^{i \xi_1 x+i\xi_2z}\int_0^\infty \l 1-e^{i\xi_2 w} \r \nu(x, dw)  \, \widehat{h}(\xi_1, \xi_2)\, d\xi_1 d\xi_2 \notag \\
 = \, & (Gh)(x,z) + \int_0^\infty \l h(x, z+w)-h(x,z) \r \nu(dw, x) \notag \\
 = \, & (Ah)(x,z)
\end{align}
and hence the operator $A$ is a pseudo-differential operator with simbol $q(x,D)-f(D,x)$. Now note that in view of \eqref{boundlev} we have that $q(x, \cdot)-f( \cdot,x)$ has bounded coefficients in the sense of \cite[eq. (2.33)]{schillinglevy}. Further since we have that $x \mapsto q(x, \cdot) $ and $x \mapsto f(\cdot, x)$ are continuous we can apply \cite[Theorem 2.36]{schillinglevy} to compute the symbol of the process $(M_t, \sigma_t)$ and we show that it is equal to $q(x,D)-f(D, x)$. We prove that
\begin{align}
t^{-1}\l \mathds{E}^xe^{-i\xi_1x-i\xi_2z}e^{i\xi_1M_t+i\xi_2\sigma^z_t}-1 \r \to  q(x,D) - f(D,x).
\end{align}
We have that
\begin{align}
& \l \mathds{E}^xe^{-i\xi_1x-i\xi_2z}e^{i\xi_1M_t+i\xi_2\sigma^z_t}-1\r \notag \\
 = \, & \l \mathds{E}^x e^{-i\xi_1x} e^{i\xi_1M_t+i\xi_2\sigma_t}-1\r \notag \\
 = \, & \l \int_{\mathbb{R}^d \times [0, \infty)} e^{i\xi_1y}  e^{i\xi_2w}  P^x \l \sigma_t \in dw \mid M_t = y \r \, P^x (M_t \in dy) -1 \r \notag \\
 = \, & \bigg( e^{-i \xi_1 x} \int_{\mathbb{R}^d \times [0, \infty)} e^{i\xi_1 y}  e^{i\xi_2 w}\int P^x \l \sigma_t \in dw \mid M_t = y, \l M_s\r_{s<t} = \omega_s \r \notag \\ & \times  P^x \l \l M_s\r_{s<t} \in d\omega_s \mid M_t=y \r P^x \l M_t \in dy \r -1 \bigg) \notag \\
 = \, &  e^{-i\xi_1x}\mathds{E}^x e^{-\int_0^t f(-i\xi_2, M_w) dw} e^{i\xi_1M_t}-1 \notag \\
 \stackrel{t \to 0+}{\sim} & e^{-i\xi_1x}\mathds{E}^x \l 1- tf(-i\xi_2, M_t) \r e^{i\xi_1 M_t}
 \label{charact}
 \end{align}
 where in the last step we have used \eqref{laplexpdef}.
 Now note that by \eqref{charact} we can write, as $t \to 0+$,
 \begin{align}
 &\lim_{t \downarrow 0} t^{-1}\l \mathds{E}^xe^{-i\xi_1x-i\xi_2z}e^{i\xi_1M_t+i\xi_2\sigma^z_t}-1\r \notag \\
= \, & \lim_{t \downarrow 0}  t^{-1} \int_{\mathbb{R}^d}\l e^{i\xi_1y-i\xi_1x} -1 \r P^x \l M_t \in dy \r \notag \\ &- e^{-i\xi_1x}\int_{\mathbb{R}^d} e^{i\xi_1y}  f(-i\xi_2, y)\, P^x \l M_t \in dy \r  \notag \\
= \,  &   q(x, \xi_1) - f(-i\xi_2, x),
\end{align}
where in the last step we have  used the fact that $\mathds{E}^xu(M_t) \to u(x)$ as $t \to 0$ for any continuous bounded function $u$ and that $q(x,D)$ is the symbol of the process $M_t$.
\end{proof}

In the forthcoming results it will be useful to know the following result.
\begin{lem}
\label{nosimjump}
The processes $M_t$ and $\sigma_t$ don't jump simultaneously, a.s.
\end{lem}
\begin{proof}
This is a consequence of \cite[(1.6)d and Remark (2.8)]{cinlarlevy}.
\end{proof}
\begin{os}
\label{remrepra}
Note that by Proposition \ref{propgen} we have that $A$ is a pseudo-differential operator whose representation is of the form \eqref{reprg} on $\mathbb{R}^{d+1}$. Precisely we have that
\begin{align}
&(Ah) (x,z) \notag \\ = \, &-c(x) h(x,z)+l(x) \cdot \nabla_x h(x,z) + \frac{1}{2} \text{div} Q(x) \nabla_x h(x,z) \notag \\
& + \int_{\mathbb{R}^{d+1}} \l h(x+y,z+w)-h(x,z)- \nabla_x h (x,z) \cdot y  \chi (|y|)  \r K(x,z,dy,dw)
\label{repra}
\end{align} 
where the jump kernel $K(x,z,dy,dw) := \delta_0(dw)N(x, dy)+\delta_0(dy)\nu(x,dw)$ is supported on the coordinate axes $\l \mathbb{R}^d \times \ll 0 \rr \r \times \l \ll 0 \rr \times [0, \infty) \r$, since the processes $M_t$ and $\sigma_t$ don't jump simultaneously, a.s., by Lemma \ref{nosimjump}. Hence
\begin{align}
&(Ah) (x,z) \notag \\ = \, &-c(x) h(x,z)+l(x) \cdot \nabla_x h(x,z) + \frac{1}{2} \text{div} Q(x) \nabla_x h(x,z) \notag \\
& + \int_{\mathbb{R}^{d}} \l h(x+y,z)-h(x,z)- \nabla_x h (x,z) \cdot y  \chi (|y|)  \r N(x,dy) \notag \\ & + \int_0^\infty \l h(x,z+w)-h(x,z) \r \nu(x, dw).
\end{align}
\end{os}
Now we obtain some properties of the operator $\Pi_t$ and the corresponding mapping $t \mapsto \Pi_tu$, which will be used in the subsequent results.
The following auxiliary lemmas characterize the strong continuity with respect to $\left\| \cdot \right\|$ of $t \mapsto \Pi_tu$ for $u \in C_0 \l \mathbb{R}^d \r$.
\begin{lem}
\label{lem:ontinuity}
Suppose that $\bar{u}(s):=\inf_{x\in \Rb^d}\bar{\nu}(s,x)$ is the tail of a \LL measure of some subordinator of infinite activity. Then, for any $\delta>0$, $s, t \geq 0$, it is true that
\begin{align}
\lim_{s \to t} \sup_{x\in \Rb^d} P^x \l \left| L_t - L_s \right| > \delta \r \, = \, 0
\end{align}
\end{lem}
\begin{proof}
	Our first aim is to construct path-wise a proper subordinator $\underline{\sigma}$ such that regardless of the initial position $x$ of $M$, stochastically $\sigma_s\geq \underline{\sigma}_s$ for all $s\geq 0$. This  can be done as follows. Consider the dyadic decomposition of $\Rb^+$, that is $\lbrb{\lbbrb{\frac{k}{2^n},\frac{k+1}{2^n}}}_{k\geq 0}$, and from $\sigma_t=\sum_{s\leq t}\Delta \sigma_s$ define for any fixed $t>0$ and the running trajectory of $M$
	\begin{equation}\label{eq:lower}
		\underline{\sigma}^{(n)}_t=\sum_{s\leq t}\sum_{k}\frac{k}{2^n}\ind{\Delta \sigma_s\in \lbbrb{\frac{k}{2^n},\frac{k+1}{2^n}}}\mathrm{B}_{M_{s-},\frac{k}{2^n}}, 
	\end{equation}
	where $\mathrm{B}_{M_{s-},\frac{k}{2^n}}$ is a Bernoulli random variable with parameter $p_{M_{s-},\frac{k}{2^n}}$ such that 
	\[p_{M_{s-},\frac{k}{2^n}}=\frac{u\lbrb{\lbbrb{\frac{k}{2^n},\frac{k+1}{2^n}}}}{\nu\lbrb{\lbbrb{\frac{k}{2^n},\frac{k+1}{2^n}}, M_{s-}}}\leq 1\]
	 provided $\nu\lbrb{\lbbrb{\frac{k}{2^n},\frac{k+1}{2^n}}, M_{s-}}>0$ and zero otherwise. Conditionally on the path of $M$ all Bernoulli random variables are independent of each other. Note that we have used that $\bar{u}(s)=\inf_{x\in \Rb^d}\bar{\nu}(x,s)$ with $\nu$ and $u$ being the respective measures behind the respective tails. We note that for any $t>0$, $n\geq 1$ conditionally on $\lbrb{M_s=w_s}_{s\leq t}$ 
	 \begin{equation}\label{eq:ineqS}
	 	\underline{\sigma}^{(n)}_t\leq \sigma_t.
	 \end{equation}
	 Moreover, 
	 \begin{equation*}
	 \begin{split}
	 &	\Eds{e^{-\lambda	\underline{\sigma}^{(n)}_t }\Big|M_s=w_s,s\leq t}=e^{-\sum_{k}\int_{0}^t\lbrb{1-e^{-\lambda \frac{k}{2^n}}}\frac{u\lbrb{\lbbrb{\frac{k}{2^n},\frac{k+1}{2^n}}}}{\nu\lbrb{M_{s-},\lbbrb{\frac{k}{2^n},\frac{k+1}{2^n}}}}\nu\lbrb{M_{s-},\lbbrb{\frac{k}{2^n},\frac{k+1}{2^n}}}ds}\\
	 &= e^{-t\sum_{k}\lbrb{1-e^{-\lambda \frac{k}{2^n}}}u\lbrb{\lbbrb{\frac{k}{2^n},\frac{k+1}{2^n}}}}   
	 \end{split}
	 \end{equation*}
	 and we see that $\underline{\sigma}^{(n)}_t$ are Compound Poisson processes and clearly 
	 \begin{equation}\label{eq:limitS}
	 \begin{split}
	 &\limi{n}\underline{\sigma}^{(n)}_t\stackrel{d}{=}\underline{\sigma}_t,
	 \end{split}
	 \end{equation}  
	 where $\underline{\sigma}$ is a subordinator with
	 \begin{equation}\label{eq:LKS}
	 \begin{split}
	 &\Eds{e^{-\lambda\underline{\sigma}_1 }}=e^{-\int_{0}^{\infty}\lbrb{1-e^{-\lambda y}}u(dy)}.
	 \end{split}
	 \end{equation}
	 From now on fix $t>0$ and $x\in \Rb^d$. We consider $0<a_l\uparrow t$. Let $\epsilon>0$. Then
	 \begin{equation*}
	 \begin{split}
	 \Pb^x\lbrb{L_t-L_{a_l}>\epsilon}&=\Pb^x\lbrb{L_t-L_{a_l}>\epsilon;\sigma_{L_{a_l}}\leq t}\\
	 &=\int_{y\in \Rb^d}\int_{a_l}^t\Pb^x\lbrb{L_t-L_{a_l}>\epsilon,\sigma_{L_{a_l}} \in dv;  M_{L_{a_l}}\in dy} \\
	 &= \int_{y\in \Rb^d}\int_{a_l}^t\Pb^{y}\lbrb{L_{t-v}>\epsilon}\Pb^x\lbrb{\sigma_{L_{a_l}} \in dv;  M_{L_{a_l}}\in dy}\\
	 &\leq\sup_{y\in \Rb^d}\Pb^y\lbrb{L_{t-a_l}>\epsilon}=\sup_{y\in \Rb^d}\Pb^y\lbrb{\sigma_{\epsilon}\leq t-a_l}.
	 \end{split}
	 \end{equation*}
	 From \eqref{eq:ineqS} we have that for any $n\geq 1$
	 \begin{equation*}
	 \begin{split}
	 &	\Pb^x\lbrb{L_t-L_{a_l}>\epsilon}\leq 	\Pb\lbrb{\sigma^{(n)}_{\epsilon}\leq t-a_l}    
	 \end{split}
	 \end{equation*}
	 as $\underline{\sigma}^{(n)}$ is independent of the initial position of the Markov process $M$. By Portmanteau's theorem we deduct that 
	 \begin{equation*}
	 \begin{split}
	 &	\Pb^x\lbrb{L_t-L_{a_l}>\epsilon}\leq 	\limsupi{n}\Pb\lbrb{\underline{\sigma}^{(n)}_{\epsilon}\leq t-a_l}\leq \Pbb{\underline{\sigma}_\epsilon\leq t-a_l}.    
	 \end{split}
	 \end{equation*}
	 Since $\underline{\sigma}$ is of infinite activity, that is $\bar{u}(0)=\infty$, we obtain that
	 \begin{equation*}
	 \begin{split}
	 &		  \limsupi{l} \sup_{x \in \Rb^d}\Pb^x\lbrb{L_t-L_{a_l}>\epsilon}\leq \limi{l}\Pbb{\underline{\sigma}_\epsilon\leq t-a_l}=0.
	 \end{split}
	 \end{equation*}
	 The other scenario when $a_l\downarrow t$ is proved in the same fashion using that $\curly{L_t\neq L_{a_l}}=\curly{\sigma_{L_t}\in\lbbrb{t,a_l}}$.
\end{proof}
\begin{lem}
\label{unifcont}
Assume that the strong Markov process $M(t)$ is a Feller process and define $\overline{M}_t:= \sup_{0 \leq s \leq t} |M_s-M_0|$. Suppose that for any $\delta>0$
\begin{align}
\lim_{t \to 0 }\sup_x P^x \l  \overline{M}_t>\delta  \r = 0.
\label{hypdist}
\end{align}
Then, under the assumption of Lemma \ref{lem:ontinuity} we have that the mapping
\begin{align}
[0, \infty) \ni t \mapsto \Pi_t u (x) \,: = \, \mathds{E}^x u(M(L(t)), \text{ for } u \in C_0 \l \mathbb{R}^d \r
\end{align}
is uniformly continuous (strongly continuous with respect to $\left\| \cdot \right\|)$.
\end{lem}
\begin{proof}
Take an arbitrary sequence $a_n \uparrow t$. Recall that $C_0 \l \mathbb{R}^d \r$ functions are uniformly continuous and hence pick $u \in C_0 \l \mathbb{R}^d \r$ and fix $\epsilon>0$ such that one has $|u(x)-u(y)|<\epsilon$ whenever $|x-y|<\delta$.  Fix another arbitrary constant $\delta^\star>0$. We have that
\begin{align}
& \left| \int \l u(M_{L_t}) - u(M_{L_{a_n}}) \r  dP^x  \right|\notag \\  \leq \, &\int \left| u(M_{L_t}) - u(M_{L_{a_n}})  \right| dP^x \notag \\
= \, & \int_{\ll L_t = L_{a_n} \rr} \l u(M_{L_t}) - u(M_{L_{a_n}}) \r dP^x  + \int_{\ll L_t-L_{a_n} > \delta^\star\rr}   \left| u(M_{L_t}) - u(M_{L_{a_n}})  \right| dP^x \notag \\
 & + \int_{\ll 0< L_t-L_{a_n} < \delta^\star, |M_{L_t} -M_{L_{a_n}}|<\delta\rr}   \left| u(M_{L_t}) - u(M_{L_{a_n}})  \right| dP^x \notag \\ 
& + \int_{\ll 0< L_t-L_{a_n} < \delta^\star, |M_{L_t} -M_{L_{a_n}}|>\delta\rr}   \left| u(M_{L_t}) - u(M_{L_{a_n}})  \right| dP^x \notag \\
\leq \, & 2 \left\| u \right\|  P^x \l L_t-L_{a_n} > \delta^\star \r + \epsilon + 2 \left\| u \right\| P^x \l   L_t-L_{a_n} < \delta^\star, |M_{L_t} -M_{L_{a_n}}|>\delta \r.
\end{align}
Recall the action of the translation operator $\theta_t$ on $L_t$, i.e., for any stopping time $\tau$
\begin{align}
L_t \circ \theta_{\tau} \, = \, \inf \ll w \geq 0: \sigma_{w+\tau}-\sigma_\tau > t \rr
\end{align}
and use \eqref{stmp} to say that
\begin{align}
& P^x \l   |M_{L_t} -M_{L_{a_n}}|>\delta,  L_t-L_{a_n} < \delta^\star \r  \notag \\ = \, & P^x \l \left| \l M_{L_{t-\sigma_{L_{a_n}}}}-M_0 \r \circ \theta_{L_{a_n}} \right| > \delta , L_{t-\sigma_{L_{a_n}}} \circ \theta_{L_{a_n}} < \delta^\star  \r \notag \\
= \, & \mathds{E}^x P^{M_{L_{a_n}}} \l |M_{L_{t-\sigma_{L_{a_n}}}} - M_0 | > \delta, L_{t-\sigma_{L_{a_n}}}< \delta^\star \r \notag \\
\leq \,&  \sup_y P^y \l \overline{M}_{\delta^\star} > \delta \r.
\end{align}
Hence we have by Lemma \ref{lem:ontinuity} that
\begin{align}
\lim_n \sup_x | \mathds{E}^x \l u(M_{L_t}) - u(M_{L_{a_n}}) \r | \, \leq \, \epsilon  + 2 \left\| u \right\| \sup_x P^x \l \overline{M}_{\delta_\star} > \delta \r
\end{align}
Now let $\delta^\star \to 0$ and use \eqref{hypdist}. Since $\epsilon$ is arbitrary and by repeating the same argument for $a_n \downarrow t$ we get the result.
\end{proof}
In the following proposition we characterize function spaces on which we want that the linear operators $\Pi_t$ and $\mathpzc{R}_\lambda := \int_0^\infty e^{-\lambda t} \Pi_t \, dt$ act.
\begin{prop}
\label{propcc}
Under the assumptions of Proposition \ref{propgen} and Lemma \ref{unifcont} further suppose that for any $\lambda>0$ there exist two positive constants $c$ and $C$ such that $c \leq f(\lambda,x) \leq C$. Let $\Pi_tu (x) := \mathds{E}^x u(M(L(t)))$ and define
\begin{align}
\mathpzc{R}_\lambda u \, : = \, \int_0^\infty e^{-\lambda t} \Pi_t u \, dt.
\end{align}
Then we have that $\Pi_t: C_0 \l \mathbb{R}^d \r \mapsto C_0 \l \mathbb{R}^d \r $ and $\mathpzc{R}_\lambda: C_0 \l \mathbb{R}^d \r \mapsto \text{Dom}(G)$.
\end{prop}
\begin{proof}
First we prove that $\mathpzc{R}_\lambda u \in C_0 \l \mathbb{R}^d \r$ for any $u \in C_0 \l \mathbb{R}^d \r$. Observe that, for any Borel set $B \subset \mathbb{R}^d$ we have that
\begin{align}
\l \mathpzc{R}_\lambda \ind B \r(x) \, = \, &\mathds{E}^x\int_0^\infty e^{-\lambda t} \ind {B}(X_t) dt \notag \\
&\mathds{E}^x \sum_y \int_{\sigma^0(y-)}^{\sigma^0(y)} e^{-\lambda t} \ind B(M_y) \, dt \notag \\
= \, & \lambda^{-1} \mathds{E}^x \sum_y \l e^{-\lambda \sigma^0(y-)} - e^{-\lambda \sigma^0(y)} \r \ind B(M_y) \notag \\
= \, & \lambda^{-1} \mathds{E}^x \sum_y \l 1-e^{-\lambda (\sigma^0(y)-\sigma^0(y-))} \r e^{-\lambda \sigma^0(y-)} \ind B(M_y).
\end{align}
Recall that under our assumptions $M$ is a Hunt process with a reference measure. Hence apply \cite[Lemma 2.24]{cinlarlevy}) and use also the fact that, a.s., $\sigma^0(y) = \sigma^0(y-)$ to write
\begin{align}
\widetilde{q}(\lambda, x) \, = \,& \lambda^{-1} \mathds{E}^x \int_0^\infty \int_0^\infty \l 1-e^{-\lambda s} \r \nu(ds, M_y) e^{-\lambda \sigma^0(y-)} \ind B(M(y)) dy \notag \\
= \, & \lambda^{-1} \mathds{E}^x \int_0^\infty  e^{-\lambda \sigma^0(y)} f(\lambda, M_y) \ind B(M(y)) dy.
\end{align}
Hence we have by a classical standard machine argument that, for any $u \in B_b \l \mathbb{R}^d \r$,
\begin{align}
\mathpzc{R}_\lambda u\, = \, \lambda^{-1} \mathds{E}^x \int_0^\infty  e^{-\lambda \sigma^0(y)} f(\lambda, M_y) u(M(y)) dy.
\label{calclapl}
\end{align}
Hence if we define $R$ as the potential operator of $P_t$, i.e.,
\begin{align}
Rh : = \int_0^\infty P_th \,  dt
\end{align}
we obtain
\begin{align}
\mathpzc{R}_\lambda u (x) \, = \,  \lambda^{-1} Rh(x,0) \text{ where } h(x,z) = u(x) f(\lambda, x) e^{-\lambda z}.
\label{respot}
\end{align}
Since $P_t$ has the Feller property, we have that $P_th \in C_0 \l \mathbb{R}^d \times [0, \infty) \r$ and further
\begin{align}
| P_t h | \, \leq \,  C \left\| u \right\| \mathds{E}e^{-\sigma^0(t)} \, = \,   C \left\| u \right\| \mathds{E}^xe^{-\int_0^t f(\lambda, M_w) dw} \, \leq \, C \left\|  u \right\| e^{-ct}
\label{boundph}
\end{align}
which is integrable on $(0, \infty)$. Hence 
\begin{align}
x \mapsto Rh(x,0) \in C_0 \l \mathbb{R}^d \r
\label{rc0}
\end{align}
by the dominated convergence theorem since $(t,x,z) \mapsto P_th(x,z)$ is continuous for each $h \in C_0 \l \mathbb{R}^d \times [0, \infty) \r$. The fact that $\Pi_tu \in C_0 \l \mathbb{R}^d \r$ follows from uniform continuity of $t \mapsto \Pi_tu$ and \cite[Proposition 1.7.6]{abhn}.

Now we prove that $\mathpzc{R}_\lambda u \in \text{Dom}(G)$ for $u \in C_0 \l \mathbb{R}^d \r$. We have that
 \begin{align}
\mathpzc{R}_\lambda u (x) \, = \,\lambda^{-1} Rh (x,0).
 \end{align}
But since $\text{Ran}(R) \subset \text{Dom}(A)$ one has that $Rh(x,z) \in \text{Dom}(A)$ and $ARh=-h \in C_0 \l \mathbb{R}^d \r$ (e.g. \cite[Lemma 3.5.72]{niels2}). Then since
\begin{align}
ARh(x,z) = \lambda e^{-\lambda z} \l  G\mathpzc{R}_\lambda u  + f(\lambda, x) \mathpzc{R}_\lambda u \r
\end{align}
and $\mathpzc{R}_\lambda u \in C_0 \l \mathbb{R}^d \r$ we have that $G\mathpzc{R}_\lambda u \in C_0 \l \mathbb{R}^d \r$. It follows that $\mathpzc{R}_\lambda  u \in \text{Dom}(G)$.
\end{proof}
Now we can provide the form of the Kolmogorov equation of $X_t$. We show that $t \mapsto \Pi_tu$ is a mild solution of \begin{align}
\mathfrak{D}^\cdot_tq(t) \, = \, G q(t), \qquad q(0) = u 
\label{eqtheorem}
\end{align}
i.e., it is a function $q(t) \in C\l [0, \infty) ; C_0 \l \mathbb{R}^d \r \r$ such that $\int_0^t q(s) ds \in \text{Dom}(G)$ which satisfies
\begin{align}
\int_0^t \l q(s, \cdot) -u(\cdot) \r \bar{\nu}(t-s, \cdot) \, ds \, = \, G \int_0^t q(s) \, ds.
\end{align}
Note that this notion of mild solution is equivalent to the classical notion (e.g. \cite[Def. 6.3]{engelnagel}) valid for the abstract Cauchy problem
\begin{align}
\frac{d}{dt} q(t) = G q(t) , \qquad q(0) = u,
\label{acp}
\end{align}
where $G$ is a closed linear operator. Indeed a function $q(t)$ is said to be a mild solution of \eqref{acp} if 
\begin{align}
\int_0^t q(s) ds \in \text{Dom}(G) \text{ and } q(t) - u \, = \, G \int_0^t q(s) \, ds.
\end{align}
It is a known fact that $t \mapsto T_tu$ is a mild solution of \eqref{acp} for any $u \in C_0 \l \mathbb{R}^d \r$ (e.g. \cite[Proposition 3.1.9]{abhn}) and thus the following theorem provides the semi-Markov analogue of this fact.
\begin{te}
\label{teeqcont}
Under the assumptions of Proposition \ref{propcc} we have that the mapping 
\begin{align}
[0, \infty) \ni t \mapsto q(t) \, := \, \Pi_tu 
\end{align}
is a mild solution of \eqref{eqtheorem} for any $u \in C_0 \l \mathbb{R}^d \r$.
\end{te}
\begin{proof}
Let
\begin{align}
v(t) \,: = \, \int_0^t q(s) \, ds,
\end{align}
and define
\begin{align}
\widetilde{v}(\lambda) \, : = \, \int_0^\infty e^{-\lambda t} v(t) \, dt.
\end{align}
We have by \cite[eq. (1.11)]{abhn} that
\begin{align}
\widetilde{v}(\lambda) \, = \, \lambda^{-1}\mathpzc{R}_\lambda u 
\label{newvermild}
\end{align}
and thus $\widetilde{v}(\lambda) \in \text{Dom}(G)$ by Proposition \ref{propcc}.
Let
\begin{align}
\widetilde{q}(\lambda) \, = \, \int_0^\infty e^{-\lambda t} q(t) \, dt
\end{align}
and then note that \eqref{calclapl} implies
\begin{align}
\widetilde{q}(\lambda,x) \, = \, \lambda^{-1} \mathds{E}^x \int_0^\infty  e^{-\lambda \sigma^0(y)} f(\lambda, M_y) u(M(y)) dy.
\end{align}
Now define
\begin{align}
(R h)(x,z) \, : = \, 	\lim_{N\to \infty} \int_0^N \mathds{E}^x h(M_y, \sigma^z_y) dy
\end{align}
and note that, for $h(x,z): = e^{-\lambda z} f(\lambda, x) u(x)$, one has
\begin{align}
\lambda^{-1}(R h)(x,0) \, = \, \widetilde{q}(\lambda, x).
\end{align}
For $ z\geq 0$ one has instead by repeating the computation \eqref{calclapl} above that 
\begin{align}
(R h)(x,z) \, = \, &  e^{-\lambda z} \mathds{E}^x \int_0^\infty  e^{-\lambda \sigma^0(y)} f(\lambda, M_y) u(M(y)) dy \notag \\
= \, & \lambda e^{-\lambda z} \widetilde{q}(\lambda, x).
\label{potanyz}
\end{align}
Use now the representation \eqref{reprgen} together with \eqref{potanyz} to say that
\begin{align}
(-ARh)(x,0) \, = \, \lambda \l f(\lambda, \cdot) -G \r \widetilde{q}(\lambda, x).
\label{341}
\end{align}
Now if $h(x, z) \in \text{Dom}(R)$, i.e.,
\begin{align}
\text{Dom}(R)= \ll h \in  C_0 \l \mathbb{R}^d \times [0, \infty) \r : Rh \in C_0 \l \mathbb{R}^d \times [0, \infty) \r \rr,
\end{align}
we could use \cite[Lemma 3.5.72]{niels2} to say that $-ARh=h$ if $Rh \in \text{Dom}(A)$, but we proved this before. Hence, use this in \eqref{341} to find that
\begin{align}
\lambda \l f(\lambda, \cdot) -G \r \widetilde{q}(\lambda, x) \ = \, h(x, 0) \, = \, f(\lambda, x) u(x).
\label{363}
\end{align}
Now we have by \eqref{newvermild}and \eqref{363} that
\begin{align}
G \widetilde{v}(\lambda) \, = \, G \frac{\widetilde{q}(\lambda, \cdot)}{\lambda} \, = \, \frac{f(\lambda, \cdot)}{\lambda} \widetilde{q}(\lambda, \cdot) - \frac{f(\lambda)}{\lambda^2} u \, =: \, \widetilde{g}(\lambda)
\end{align}
and since
\begin{align}
\int_0^\infty e^{-\lambda t} \bar{\nu}(t,x) dt \, = \, \frac{f(\lambda,x)}{\lambda}
\end{align}
we have by \cite[Proposition 1.6.4]{abhn} that
\begin{align}
g(t) \, = \, \int_0^t \l q(s,\cdot) - u \r \, \bar{\nu}(t-s,\cdot) \, ds.
\end{align}
Note that since by Lemma \ref{unifcont} we know that $t \mapsto q(t,x)$ is $C \l [0, \infty) ; C_0 \l \mathbb{R}^d \r \r$  we have that $t \mapsto g(t)$ is $C \l [0, \infty) ; C_0 \l \mathbb{R}^d \r \r$ by \cite[Proposition 1.3.4]{abhn} and thus, since $G$ is closed, it follows from \cite[Proposition 1.7.6]{abhn} that $Gv(t) = \int_0^t \l q(s,\cdot) - u \r \, \bar{\nu}(t-s,\cdot) \, ds$ for all $t \geq 0$, i.e., the function $q(t, \cdot)$ is a mild solution of \eqref{eqtheorem}.
\end{proof}
In the forthcoming Section \ref{secconv} we will study the asymtpotic behaviour of $X(t)=B \l L(t) \r$ as $t \to \infty$. It turns out that some interesting and clarifying examples concerns the case in which $x \mapsto f(\lambda, x)$ is a stepped function. Hence, for completeness, we provide the form of the Kolmogorov's equation of $X_t$ to cover the case in which $x \mapsto f(\lambda, x)$ is not continuous. In this case, for example, the assumptions of Proposition \ref{propgen} are not satisfied as well as the assumptions of Proposition \ref{propcc}. An inspection of the proof shows indeed that $\mathpzc{R}_\lambda u \in \text{Dom}(G)$ is a consequence of $P_t h \in C_0 \l \mathbb{R}^d \r$ and that $h \in \text{Dom}(R)$ where $h(x,z) = u(x) e^{-\lambda z} f(\lambda, x)$. In general this is no more true, even if we equip $P_t$ with the Feller property since $x \mapsto f(\lambda, x)$ is not continuous.

However it turns out that the equation can be still written down as in Theorem \ref{teeqcont} in an approximate sense.
Hence we will consider an approximating sequence of Bernstein functions $f^n(\lambda, x)$, each one of which satisfies the assumptions used above and such that $f^n (\lambda, x) \to f(\lambda, x)$ where $x \mapsto f(\lambda, x)$ is not necessarily continuous but satisfies the assumption \eqref{boundlev} and $f^n(\lambda, x)$ is bounded above and below by constants $c_n$ and $C_n$ (which can of course depend on $\lambda$).  Hence we need to provide first a weak convergence result. 
\begin{prop}
\label{conv}
Assume that $x \mapsto f(\lambda, x)$ is bounded below by $c>0$ and above by $C\geq c$ and that \eqref{boundlev} holds. Assume further that there exists a sequence $f^n \l \lambda, x \r \to f(\lambda, x)$ such that, for any $n$ the functions $f^n \l \lambda, x \r$ are Bernstein functions as in Propositions \ref{propgen} and \ref{propcc} with constants $c_n \leq f^n(\lambda, x) \leq C_n$. Assume that $\inf c_n >0$, $\sup C_n<\infty$. Now let $\bar{u}^n(s)$ be the tail of Lemma \ref{lem:ontinuity} and assume that $\mathfrak{u}(s):=\inf_n \bar{u}^n(s)$ is the tail of the L\'evy measure of a subordinator with infinite activity. Let $M$ be a Feller process as in Lemma \ref{unifcont}.
Let $\Pi_t^n u:= \mathds{E}^x u \l M(L^n(t)) \r$. Then for any $u \in C_b \l \mathbb{R}^d \r$ one has $\Pi_t^nu \to \Pi_tu$. 
\end{prop}
\begin{proof}
Let
\begin{align}
\mathpzc{R}_\lambda^n u \, : = \, \int_0^\infty e^{-\lambda t} \, \Pi_t^n u \, dt.
\end{align}
and recall that $\left| \Pi_t^nu \right| \leq \left\|u \right\|$.
Let's apply again \cite[Proposition 1.7.6 and Theorem 1.7.3]{abhn} to say that if $\mathpzc{R}_\lambda^n u  \to \mathpzc{R}_\lambda u$ then $\Pi_tu = \lim_n \Pi_t^n u$ for almost all $t \geq 0.$ However if $t \mapsto \lim_n \Pi_t^nu$ is continuous we have that the equality is true for any $t\geq 0$ since $t \mapsto \Pi_tu$ is certainly continuous by Lemma \ref{unifcont}.
First we prove that $\mathpzc{R}_\lambda^n u  \to \mathpzc{R}_\lambda u$ and then we prove the continuity of $t \mapsto \lim_n\Pi_t^nu$.
By \eqref{calclapl} we obtain that
\begin{align}
\mathpzc{R}_\lambda^n u (x) \, = \, &\lambda^{-1} \int_0^\infty \mathds{E}^x e^{-\lambda \sigma_n^0(t)} u(M(t)) f^n(\lambda, M (t)) dt \notag \\
= \, & \lambda^{-1} \int_0^\infty P_t^nh^n(x,0) \, dt
\end{align}
where $h^n(x,z) = f^n(\lambda,x) u(x) e^{-\lambda z}$ and $P_t^n$ denote the Feller semigroup of the process $\l M_t, \sigma_t^n \r$. Note now that
\begin{align}
&P_t^nh^n (x,0) \notag \\ = \, &\int_{\mathbb{R}^d } \int_0^\infty f^n(\lambda, y) \, u(y) e^{-\lambda z} P^x \l M(t) \in dy, \sigma_n^0(t) \in dz \r  \notag \\
= \, & \int_{\mathbb{R}^d }f^n(\lambda, y) \, u(y) \mathds{E}^x \left[  e^{-\int_0^t f^n(\lambda, M_s)ds}  \mid   M(t) = y\right] P^x \l M(t) \in dy \r \notag \\
\longrightarrow \, &  \int_{\mathbb{R}^d }f(\lambda, y) \, u(y) \mathds{E}^x \left[  e^{-\int_0^t f(\lambda, M_s)ds}  \mid   M(t) = y\right] P^x \l M(t) \in dy \r \notag \\
= \, & P_th(x,0)
\end{align}
where in the last step the limit is moved inside the integrals by the bounded convergence theorem.
Further note
\begin{align}
| P_t^nh^n (x,0) | \, \leq  \,& \l \sup C_n \r\left\| u \right\| \mathds{E}e^{-\int_0^t f^n(\lambda, M_y)dy} \notag \\
\leq \, &\l \sup C_n \r \left\| u \right\| e^{-t\inf_n c_n}
\end{align}
which is integrable and thus $\mathpzc{R}_\lambda^n u  \to \mathpzc{R}_\lambda u$ by the dominated convergence theorem.
Now we prove that $t \mapsto \lim_n \Pi_t^nu$ is continuous by showing that $t \mapsto \Pi_t^nu$ is uniformly continuous  with respect to $n$. This can be done, under the assumption $\bar{u}(0,x):= \inf_n \bar{\nu}^n(0,x) = \infty$ for any $x \in \mathbb{R}^d$ with the same argument used in Lemma \ref{unifcont}. Hence we can construct the subordinator $\underline{\sigma}$ such that stochastically $\sigma_s^n \geq \underline{\sigma}_s$ regardless of $n$ (and $x$) as in Lemma \ref{lem:ontinuity}. Hence define for any fixed $t>0$ and the running trajectory of $M$
	\begin{equation}\label{eq:lowerwn}
		\underline{\sigma}^{(m,n)}_t=\sum_{s\leq t}\sum_{k}\frac{k}{2^m}\ind{\Delta \sigma_s\in \lbbrb{\frac{k}{2^m},\frac{k+1}{2^m}}}\mathrm{B}^n_{M_{s-},\frac{k}{2^n}}, 
	\end{equation}
where $\mathrm{B}^n_{M_{s-},\frac{k}{2^n}}$ is a Bernoulli random variable with parameter $p^n_{M_{s-},\frac{k}{2^n}}$ such that 
	\[p^n_{M_{s-},\frac{k}{2^m}}=\frac{\mathfrak{u}\lbrb{\lbbrb{\frac{k}{2^m},\frac{k+1}{2^m}}}}{\nu^n\lbrb{M_{s-},\lbbrb{\frac{k}{2^m},\frac{k+1}{2^m}}}}\leq 1\]
	 provided $\nu^n\lbrb{\lbbrb{\frac{k}{2^m},\frac{k+1}{2^m}},M_{s-}}>0$ and zero otherwise.
Hence conditionally on $\lbrb{M_s=w_s}_{s\leq t}$ 
	 \begin{equation}\label{eq:ineqSwn}
	 	\underline{\sigma}^{(m,n)}_t\leq \sigma_t^n.
	 \end{equation}
	 Moreover, 
	 \begin{equation*}
		\Eds{e^{-\lambda	\underline{\sigma}^{(m,n)}_t }\Big|M_s=w_s,s\leq t}= e^{-t\sum_{k}\lbrb{1-e^{-\lambda \frac{k}{2^n}}}\mathfrak{u}\lbrb{\lbbrb{\frac{k}{2^m},\frac{k+1}{2^m}}}}   
	 \end{equation*}
	 and $\underline{\sigma}^{(m,n)}_t$ are compound Poisson processes and such that $\limi{m}\underline{\sigma}^{(m,n)}_t\stackrel{d}{=}\underline{\sigma}_t$, where $\underline{\sigma}$ is a subordinator with $\Eds{e^{-\lambda\underline{\sigma}_1 }}=e^{-\int_{0}^{\infty}\lbrb{1-e^{-\lambda y}}\mathfrak{u}(dy)}$. As in Lemma \ref{lem:ontinuity} one has
	   \begin{equation*}
	 \Pb^x\lbrb{L_t^n-L_{a_l}^n>\epsilon} \leq  \sup_{y\in \Rb^d}  \Pb^y\lbrb{\sigma_{\epsilon}^n\leq t-a_l} \, \leq \, P \l \underline{\sigma}_\epsilon^{(m,n)} \leq t-a_l \r
	 \end{equation*}
and thus by the same argument we have \begin{equation*}
	 \limsupi{l} \sup_n \sup_{x \in \Rb^d}\Pb^x\lbrb{L_t^n-L_{a_l}^n>\epsilon}\leq \limi{l}\Pbb{\underline{\sigma}_\epsilon\leq t-a_l}=0.
	 \end{equation*}
Now we can repeat the same steps as in Lemma \ref{unifcont} to say that $t \mapsto \Pi_t^nu$ is continuous uniformly in $n$.

\end{proof}
Here we provide the approximation of the Kolmogorov's equation of $X(t)$, i.e., we show that $t \mapsto \Pi_tu$ is a mild solution of \eqref{eqtheorem} in the following sense
\begin{align}
\int_0^t \l \Pi_su(x) -u(x) \r \, \bar{\nu}(t-s,x) \, ds \, = \, \lim_n G \int_0^t \Pi_s^nu \, ds.
\label{approxmild}
\end{align}
\begin{prop}
\label{teeqbound}
Let $M$ be the Feller process generated by $\l G, \text{Dom}(G)\r$ under the assumptions on Theorem \ref{teeqcont}. Suppose that $f$ satisfies the assumptions of Proposition \ref{conv} and further that $\int_0^t \sup_n \bar{\nu}^n(s,x)ds < \infty$. Denote $\Pi_t^nu := \mathds{E}^x u \l M \l L^n_t \r \r$. Then we have that the mapping
\begin{align}
[0, \infty) \ni t \mapsto  q(t):=\Pi_t u
\end{align}
is a mild solution of \eqref{eqtheorem} in the sense of \eqref{approxmild}.
\end{prop}
\begin{proof}
We have by Theorem \ref{teeqcont} that $\Pi_t^nu$ satisfies
\begin{align}
\int_0^t \l \Pi_s^nu(x)-u(x) \r \, \bar{\nu}^n(t-s,x) \, ds \, = \, G\int_0^t  \Pi_s^nu \, ds.
\end{align}
Since
\begin{align}
\left\| \Pi_t^nu -u \right\| \, \leq \,  2\left\| u \right\|
\end{align}
and since $\int_0^t \sup_n  \bar{\nu}^n(t-s,x)ds< \infty$ we have by the dominated convergence theorem and Proposition \ref{conv} that
\begin{align}
\int_0^t \l \Pi_su(x) -u(x) \r \, \bar{\nu}(t-s,x) \, ds \, = \, \lim_n G\int_0^t \Pi_s^nu \, ds.
\end{align}
\end{proof}
\begin{os}
Suppose that $\alpha^n(x)$ is a sequence of functions such that there exist $\delta_1>0$ and $\delta_2>0$ small enough such that $\alpha^n(x) < 1-\delta_1$ and $\alpha^n(x)> \delta_2$, so they never reach the boundary $0$ or $1$ for any $x$ and $n$. Define
\begin{align}
\nu^n(ds,x) \, = \, \frac{\alpha^n(x)s^{-\alpha^n(x)-1}}{\Gamma(1-\alpha^n(x))}ds
\end{align}
so that
\begin{align}
\bar{\nu}^n(s,x) \, = \, \frac{s^{-\alpha^n(x)}}{\Gamma(1-\alpha^n(x))}
\end{align}
and
\begin{align}
f^n(\lambda, x) \, = \, \lambda^{\alpha^n(x)}.
\end{align}
Suppose $\alpha^n(x) \to \alpha(x)$ pointwise. Then one has that $\bar{\nu}^n$ and $f^n$ satisfy the assumption of Proposition \ref{teeqbound} since $\lambda^{\alpha^n(x)}$ is always included between to constants (depending on $\lambda)$ $c_n$ and $C_n$, such that $\inf c_n >0$ as well as $ \sup C_n<\infty$ and
\begin{align}
\inf_n \inf_x \bar{\nu}^n(s,x) \, = \, C \l s^{-\underline{\alpha}} \ind{s \leq 1} + s^{-\overline{\alpha}} \ind{s>1}\r
\end{align}
where $\underline{\alpha}= \inf_{(x,n)} \alpha^n(x)$ and $\overline{\alpha} = \sup_{(x,n)} \alpha^n(x)$, is the tail of a L\'evy measure with infinite activity.
\end{os}

\section{The variable order diffusion equation and the anomalous aggregation phenomenon}
\label{secconv}
In this section we study the asymptotic behaviour of the process $X(t)=M(L(t))$ in case the leading process $M$ is a one-dimensional Brownian motion. Hence let $G = \frac{1}{2} \partial^2_x= \Delta$ and assume that
\begin{align}
\nu(ds, x) \, = \, \frac{\alpha(x)s^{-\alpha(x)-1}}{\Gamma(1-\alpha(x))}, \qquad \alpha (x) \in (0,1).
\end{align}
The equation in Theorem \ref{teeqcont} yields to
\begin{align}
\frac{d^{\alpha(x)}}{dt^{\alpha(x)}}q(t,x) \, = \, \Delta q(t,x).
\label{fracdif}
\end{align}
In Fedotov and Falconer \cite{fedofalco} the authors considered the Fokker-Plank (forward) equation
\begin{align}
\frac{\partial}{\partial t} p(x,t) \, = \, \frac{\partial^2}{\partial x^2} D \frac{\partial^{\alpha(x)}}{\partial t^{\alpha(x)}} p(x,t)
\label{forwfedfal}
\end{align}
which can be viewed as the forward equivalent of \eqref{fracdif} (for details on the relationships between equations, see also Ricciuti and Toaldo \cite[Section 5]{rictoa} and we suggest the instructive discussion in Straka \cite{strakavariable} which fully justifies the meaning of \eqref{forwfedfal} as a model for diffusive phenomena). They showed that a random walk on a lattice, which approximates the model behind \eqref{forwfedfal} as $t \to \infty$, converges in probability to the point at which $\alpha(x)$ has its minimum. They further ran some numerical simulations to validate their results.
It turns out that the small value of the anomalous exponent completely dominates the long-time behaviour of subdiffusive system. The authors refer to this phenomenon as a ``Black Swan" (term proposed by Taleb \cite{taleb}), to describe the crucial role of rares event with extreme impact. Similar aggregation phenomena where also observed for a symmetrical random walk by Fedotov \cite{fedopre}.

In this section this phenomenon is investigated rigorously for the semi-Markov process (time-changed Brownian motion) which is related to \eqref{fracdif} by the results in the previous section. Essentially our investigations validate the simulations in \cite{fedofalco} under some technical assumptions on $x \mapsto \alpha(x)$. To be precise we discuss the asymptotic behaviour of two quantities, that is
\begin{equation}\label{eq:asympRel}\frac{\int_{0}^{t}\ind{X(s)\in A}ds}t\quad\text{ and }\quad\Pbb{X(t)\in A},
\end{equation}
where $A\subseteq\Rb$ is usually a neighbourhood of the set where $\alpha$ attains minimum. Depending on the behaviour of $\mathpzc{l}\lbrb{A\cap\lbbrbb{-x,x}}$ as $x\to\infty$ we provide a criterion based on $\alpha^*=\min_{x\in\Rb}\alpha(x),\alpha_I=\lim_{x\to\infty}\alpha(x),\alpha_J=\lim_{x\to-\infty}\alpha(x)$ which distinguishes, apart from a critical case, the two-regime behaviour that is
\[\limi{t}\frac{\int_{0}^{t}\ind{X(s)\in A}ds}t\in\curly{0,1}.\]
When the function $\alpha$ attains minimum on union of intervals we are able to determine whether $\limi{t}\Pbb{X(t)\in A}$ tends to $0$ or $1$ thereby mathematically confirming the outcome of \cite{fedofalco}. We wish the stress that the existence of a limit for the first relation in \eqref{eq:asympRel} does not necessarily imply the existence of a limit for the second. We believe this to be the case in this setting but have not been able to establish this in complete generality. We also believe that this ``aggregation phenomenon'' can be shown also for other Feller processes, e.g., a stable process, and thus further investigation in this direction are needed.

We start with the introduction of some notation. For any set $A\subseteq \Rb$ we set 
\begin{equation}\label{eq:H}
	H_t(A)=\int_{0}^{t}\ind{B_s\in A}ds=\mu_{B,[0,t]}(A).
\end{equation}
For brevity we shall use $H_t:=	H_t(A)$ when $A$ is clear. Then, if $\mathpzc{l}\lbrb{\partial A}=0<\mathpzc{l}(A)$ then it holds, without a loss of generality, that
\begin{equation}\label{eq:decom}
	\sigma\lbrb{s}=\sigma_1\lbrb{H_s}+\sigma_2\lbrb{s-H_s},
\end{equation}
where $\sigma_1,\sigma_2$ are two independent increasing processes constructed from $\sigma$ as follows
\begin{equation}\label{eq:constr}
\begin{split}
&\sigma_1\lbrb{H_s}=\sum_{v\leq s}\lbrb{\sigma(v)-\sigma(v-)}\ind{B_v\in A};\,\,\sigma_2\lbrb{s-H_s}=\sum_{v\leq s}\lbrb{\sigma(v)-\sigma(v-)}\ind{B_v\notin A}.
\end{split}
\end{equation}
 Denote next $A^+=A\cap \Rb^+, A^-=A\cap \Rb^-$ and assume for the time being that $A=A^+$. Also we introduce 
\begin{equation}\label{eq:G}
	G(t):=\int_{0}^{t}\mathpzc{l}\lbrb{A\cap\lbbrbb{0,x}}dx \,\,\text{ and }\,\,D(s)=\inf\curly{t>0:G(t)>s}.
\end{equation}
 Reserve $\tau$ for the inverse local time at zero of the Brownian motion $B$. It is well-known that $\tau$ is a stable subordinator of index $1/2$. Furthermore, from \cite[Chapter 9]{bertoins}  we have that 
 \begin{equation}\label{eq:chi}
 \chi(t):=H_{\tau(t)}=\int_{0}^{\tau(t)}\ind{B(s)\in A}ds
 \end{equation}
  is a driftless subordinator with \LL measure say $\Pi_{\chi}$ and Laplace exponent $\Phi_{\chi}(u)=-\log\Eds{ e^{-u\chi(1)}},u\geq 0$. Since we use extensively two results on the growth of subordinators, see \cite[Chapter III, Theorems 13 and 14]{bertoinb}, we state them here for convenience. Some general and recent results on the growth of \LL processes can be found in  \cite{aurzada,savov09}. 
  \begin{te}\label{thm:growth}
  	Let $\zeta$ be a real valued subordinator with Laplace exponent $\Phi_{\zeta}(u)=-\log\Eds{e^{-u\zeta(1)}},u\geq 0,$ and $\Eds{\zeta(1)}=\infty$. Then the following growth estimates are valid: 
  	\begin{enumerate}
  		\item\label{it:1}\label{it:limsup} If $h:\lbrb{0,\infty}\mapsto \Rb^+$ is a function such that $h(t)/t$ increases then a.s.
  		\begin{equation}\label{eq:limsup}
  		\limsupi{t}\frac{\zeta(t)}{h(t)}=\infty\iff  \int_{1}^{\infty}\bar{\Pi}_{\zeta}\lbrb{h(t)}dt=\infty,
  		\end{equation}
  		where $\bar{\Pi}_{\zeta}(x)=\int_{x}^{\infty}\Pi_{\zeta}(dy)$. If any of the conditions fails then one has  $\limi{t}\zeta(t)/h(t)=0$.
  		\item\label{it:2} If $\Phi_{\zeta}$ is regularly varying at zero with index $\alpha\in\lbrb{0,1}$ then there is a deterministic regularly varying function of index $1/\alpha$, say $f_\zeta$, such that
  			\begin{equation}\label{eq:liminf}
  		\liminfi{t}\frac{\zeta(t)}{f_\zeta(t)}=1.
  		\end{equation}
  	\end{enumerate}
  \end{te}
An immediate corollary is the result.
\begin{coro}\label{cor1:growth}
	If $\zeta$ is stable subordinator of index $\alpha\in\lbrb{0,1}$ then, for any $\epsilon>0$ small enough, almost surely 
	\begin{equation}\label{eq1:growth}
	\begin{split}
	&\liminfi{t}\frac{\zeta(t)}{t^{\frac{1}{\alpha}-\epsilon}}=\infty;\qquad \limsupi{t}\frac{\zeta(t)}{t^{\frac{1}{\alpha}+\epsilon}}=0.
	\end{split}
	\end{equation}
\end{coro}
  With the help of these well-known results we can get the following growth result for the occupation measure $H$.
\begin{prop}\label{prop:Bounds}
	If $\chi(1)$ has a finite mean or the Laplace exponent $\Phi_\chi$ is regularly varying  at zero of index $\alpha\in\lbrb{0,1}$, then, for any $\varepsilon>0$ small enough a.s. 
	\begin{equation}\label{eq:Bounds}
		\limi{t}\frac{H_{\tau(t)}}{H_{t^{2+\varepsilon}}}=0;\qquad \limi{t}\frac{H_{\tau(t)}}{H_{t^{2-\varepsilon}}}=\infty.
	\end{equation}
\end{prop}
\begin{proof}
	From \eqref{eq1:growth} and the fact that $\tau$ is a stable subordinator of index $1/2$ we get almost surely
	\begin{equation}\label{eq:tauLower}
		\liminfi{t}\frac{\tau(t)}{t^{2-\varepsilon}}=\infty 
	\end{equation}
	and 
	\begin{equation}\label{eq:tauUpper}
	\limsupi{t}\frac{\tau(t)}{t^{2+\varepsilon}}=0. 
	\end{equation}
	The proof of \eqref{eq:Bounds} then follows by a pathwise argument in the following fashion. Set $u(t)=\max\curly{s>0:\tau(s)\leq t^{2+\epsilon}}$. Then
	\begin{equation*}
	\begin{split}
	\frac{H_{\tau(t)}}{H_{t^{2+\varepsilon}}}\leq \frac{H_{\tau(t)}}{H_{\tau(u(t)-)}}=\frac{\chi(t)}{\chi(u(t)-)}.
	\end{split}
	\end{equation*}
	From \eqref{eq:tauUpper} which holds for any $\varepsilon>0$, we conclude that there exists $\delta=\delta(\varepsilon)>0,$  such that $u(t)/t^{1+\delta}\to\infty,t\to\infty$. Therefore,
	\begin{equation*}
	\begin{split}
	&\limsupi{t}	\frac{H_{\tau(t)}}{H_{t^{2+\varepsilon}}}\leq 	\limsupi{t}	\frac{\chi(t)}{\chi(t^{1+\delta})}.    
	\end{split}
	\end{equation*}
	Now, the first relation of \eqref{eq:Bounds} follows from the strong law of large numbers, when $\Eds{\chi(1)}<\infty$ and  from  Theorem \ref{thm:growth}, when $\Phi_\chi$ is regularly varying of index $\alpha\in\lbrb{0,1}$ at zero. Indeed in the latter case there is a deterministic $f_\chi$ regularly varying of index $1/\alpha$ such that \eqref{eq:liminf} holds. Thus
 	\begin{equation*}
 	\begin{split}
 	&\limsupi{t}	\frac{H_{\tau(t)}}{H_{t^{2+\varepsilon}}}\leq 	\limsupi{t}	\frac{\chi(t)}{\chi(t^{1+\delta})}\leq 2\limsupi{t}	\frac{\chi(t)}{f_{\chi}(t^{1+\delta})}.    
 	\end{split}
 	\end{equation*}
 	Now since $g_\chi(t):=f_{\chi}(t^{1+\delta})$ is regularly varying of index $1/\alpha+\delta/\alpha$ we can use the Potter's bounds, see \cite[Theorem 1.5.6 (iii)]{bingham}, to ensure it is true that $g_\chi(t)\geq C t^{1/\alpha+\delta/\alpha-c},  0<c<\delta/\alpha, C\in\lbrb{0,\infty}$, and therefore
 	\begin{equation*}
 	\begin{split}
 	&\limsupi{t}	\frac{H_{\tau(t)}}{H_{t^{2+\varepsilon}}}\leq 	\limsupi{t}	\frac{\chi(t)}{\chi(t^{1+\delta})}\leq 2\limsupi{t}	\frac{\chi(t)}{f_{\chi}(t^{1+\delta})}\\
 	&\leq \frac{4}{C}\limsupi{t}	\frac{\chi(t)}{t^{\frac{1}{\alpha}+\frac{\delta}{\alpha}-c}}.    
 	\end{split}
 	\end{equation*}
 The fact that the last equals zero in turn follows from  Theorem \ref{thm:growth}\eqref{it:1} applied with $h(t)=t^{1/\alpha+\delta/\alpha-c}$ where in relation \eqref{eq:limsup} we have that
 $\int_{1}^\infty\bar{\Pi}_{\chi}\lbrb{h(t)}dt<\infty$ since $\Phi_{\chi}$ being regularly varying of index $\alpha\in\lbrb{0,1}$ at zero implies that $\bar{\Pi}_{\chi}\lbrb{h(t)}$ is regularly varying of index $-1-\delta+c\alpha<-1$ at infinity, see \cite[Chapter III.1]{bertoinb}.
 
 The second relation of \eqref{eq:Bounds} follows a similar pattern. Noting that with  $u(t)=\max\curly{s>0:\tau(s)\leq t^{2-\epsilon}}$ we have that $\tau(u(t))>\tau(t^{2-\epsilon})$ we arrive for some $\delta=\delta(\epsilon)>0$ at
 \begin{equation*}
 \begin{split}
 &\limsupi{t}	\frac{\chi(t)}{\chi(t^{1-\delta})}\leq \limsupi{t}	\frac{H_{\tau(t)}}{H_{t^{2-\varepsilon}}}.
 \end{split}
 \end{equation*}
 The arguments then proceed as in the previous case.
 \end{proof}
Next, let us consider two cases which distinguish between the scenario when $A$ is bounded or not. 
\subsection{Bounded set}
Since none of the asymptotic relations in \eqref{eq:asympRel} depends on finite time horizon we can assume that $A\subseteq\Rb^+$ (the Brownian motion would pass below $A$ for a finite period of time) and $\limi{x}\mathpzc{l}\lbrb{A\cap\lbbrbb{0,x}}=a\in\lbrb{0,\infty}$. In this case in the notation of \cite[Chapter 9]{bertoins} as $t\to\infty$
\[G(t)=\int_{0}^{t}\mathpzc{l}\lbrb{A\cap\lbbrbb{0,x}}dx\sim at\]
 and thus as $t\to\infty$
 \[D(t)=\inf\curly{s>0:G(s)>t}\sim \frac{t}a, \] 
 see \eqref{eq:G}. Then, according to \cite[Chapter 9, Corollary 9.4 (ii)]{bertoins} we have that 
\[\int_{0}^t \bar{\Pi}_{\chi}(x)dx\asymp \frac{t}{D(t)}\simi a, \]
where $\bar{\Pi}_{\chi}(x)=\int_{x}^{\infty}\Pi_{\chi}(dy)$. This means that $\Eds{\chi(1)}<\infty$ and therefore a.s. $\chi(t)=H_{\tau(t)}\simi \Eds{\chi(1)}t$. From Proposition \ref{prop:Bounds} we arrive at the following result.
\begin{coro}\label{cor:Bounds}
	If $A\subseteq\Rb^+$ and $\limi{x}\mathpzc{l}\lbrb{A\cap\lbbrbb{0,x}}=a\in\lbrb{0,\infty}$ then a.s.
	\begin{equation}\label{eq:Bounds1}
	\limi{t}\frac{H_t}{t^{\frac{1}{2}-\epsilon}}=\infty;\qquad \limi{t}\frac{H_t}{t^{\frac{1}{2}+\epsilon}}=0.
	\end{equation}
\end{coro}
\begin{proof}
	Since $H_{\tau(t)}\simi\Eds{\chi(1)}t$ we can use this relation in \eqref{eq:Bounds} and change variables therein.
\end{proof}
Having established sufficiently precise asymptotic behaviour of the occupation measure the next aim is to find under what conditions $\sigma_1(H_t)$ or $\sigma_2(t-H_t)$ can be compared to $\sigma_1,\sigma_2$ at deterministic times. From Corollary \ref{cor:Bounds} we  arrive at the following result.
\begin{coro}\label{cor:Bounds1}

	It holds true that
\begin{equation}\label{eq:sigma1}
\limi{t}\frac{\sigma_1\lbrb{H_t}}{\sigma_1(t^{\frac{1}{2}+\epsilon})}=0;\quad \limi{t}\frac{\sigma_1\lbrb{H_t}}{\sigma_1(t^{\frac{1}{2}-\epsilon})}=\infty,\,\,\text{a.s.},
\end{equation}
provided there exists $\alpha\in\lbrb{0,1}$ such that for any $\epsilon_1>0$ small enough a.s.
\begin{equation}\label{eq:condi}
\begin{split}
	&\limsupi{t}\frac{\sigma_1(t)}{\sigma^{\alpha}(t)}<\infty\\
	&\liminfi{t}\frac{\sigma_1(t)}{\sigma^{\alpha+\epsilon_1}(t)}>0,
\end{split}
\end{equation} 
where $\sigma^\alpha$ stands for a suitable stable subordinator of index $\alpha \in \lbrb{0,1}$ defined on the same path space as $\sigma_1$.
\end{coro}
\begin{proof}
	If \eqref{eq:condi} holds true then a.s. for some constant $C\in\lbrb{0,\infty}$ depending on the path and $\epsilon_1>0$
	\begin{equation*}
	\begin{split}
	&	\limi{t}\frac{\sigma_1\lbrb{H_t}}{\sigma_1(t^{\frac{1}{2}+\epsilon})}\leq C\limsupi{t}\frac{\sigma^{\alpha}\lbrb{H_t}}{\sigma^{\alpha+\epsilon_1}\lbrb{t^{\frac{1}{2}+\epsilon}}}. 
	\end{split}
	\end{equation*}
	From Corollary \ref{cor1:growth} we conclude that for any $\epsilon_2,\epsilon_3$ positive and small enough
	\begin{equation*}
	\begin{split}
	&	\limi{t}\frac{\sigma_1\lbrb{H_t}}{\sigma_1(t^{\frac{1}{2}+\epsilon})}\leq C\limsupi{t}\frac{\lbrb{H_t}^{\frac{1}{\alpha}+\epsilon_2}}{\lbrb{t^{\frac{1}{2}+\epsilon}}^{\frac{1}{\alpha+\epsilon_1}-\epsilon_3}}. 
	\end{split}
	\end{equation*}
	Now, for fixed $\epsilon>0$ we can choose $\epsilon_i,i=1,2,3,$ so small that for given $\epsilon_5>0$ small enough
	\[\frac{\lbrb{\frac{1}{2}+\epsilon}\lbrb{\frac{1}{\alpha+\epsilon_1}-\epsilon_3}}{\lbrb{\frac{1}{\alpha}+\epsilon_2}}=\frac{1}{2}+\epsilon_5.\]
	Thererefore, from the second relation in \eqref{eq:Bounds1} we arrive at
	\begin{equation*}
	\begin{split}
	&	\limi{t}\frac{\sigma_1\lbrb{H_t}}{\sigma_1(t^{\frac{1}{2}+\epsilon})}\leq C\limsupi{t}\lbrb{\frac{H_t}{t^{\frac{1}{2}+\epsilon_5}}}^{\frac{1}{2}+\epsilon_2}=0. 
	\end{split}
	\end{equation*}
	This proves the first limit in \eqref{eq:sigma1} and the second follows in the same manner.
\end{proof}
 Here and hereafter for any stochastic process $Y=\lbrb{Y_t}_{t\geq 0}$ with paths that are a.s. right-continuous with left limits  we use $\lbrb{\Delta Y}_{t\geq 0}:=\lbrb{Y_t-Y_{t-}}_{t\geq 0}$ for the jump process related to $Y$. We need the following elementary result.
\begin{prop}\label{prop:Close}
	Let $\sigma(t)=V(t)+Y(t),$ where $V,Y$ are two  non-decreasing processes. If $L(t)$ is again the passage time of $\sigma$ across $t>0$ and $V(L(t))/\sigma(L(t))\to 1$ in distribution then for every $\eta\in\lbrb{0,1}$ 
	\begin{align}\label{eq:Close}
	&\limi{t}\Pbb{\curly{\Delta Y(L(t))=\Delta \sigma(L(t))}\cup\curly{ \sigma(L(t))=\sigma(L(t)-)};V(L(t)-)\leq\lbrb{1-\eta}t} \notag \\
	& =0.
	\end{align}
	If $V(L(t))/\sigma(L(t))\to 1$ almost surely then almost surely, for any $\eta\in\lbrb{0,1}$,
	\begin{equation}\label{eq:Close11}
	\begin{split}
	&\sup\curly{t>0:V(L(t)-)\leq (1-\eta)t;\Delta Y(L(t))=\Delta \sigma(L(t))\cup \sigma(L(t))=\sigma(L(t)-)}\\
	&<\infty.
	\end{split}
	\end{equation}
\end{prop}
\begin{proof} The proof that follows is rather trivial. We observe that since on the event $\curly{\Delta Y(L(t))=\Delta \sigma(L(t))}\cup\curly{\sigma(L(t))=\sigma(L(t)-)}$ it is true that $V(L(t))=V(L(t)-)$ then
	\begin{equation}\label{eq:Close1}
	\begin{split}
	&\Pbb{\curly{\Delta Y(L(t))=\Delta \sigma(L(t))}\cup\curly{\sigma(L(t))=\sigma(L(t)-)};V(L(t))\leq\lbrb{1-\eta}t}\\
	&\leq \Pbb{V(L(t))\leq\lbrb{1-\eta}t}\\
	&\leq \Pbb{\frac{V(L(t))}{\sigma(L(t))}\leq 1-\eta}.
	\end{split}
	\end{equation}
	The result now follows from the assumption that $V(L(t))/\sigma(L(t))\to 1$ in distribution. Relation \eqref{eq:Close11} follows from the fact that on  \[\curly{V(L(t)-)\leq (1-\eta)t;\curly{\Delta Y(L(t))=\Delta \sigma(L(t))}\cup\curly{\sigma(L(t))=\sigma(L(t)-)}}\] we have that 
	\[\frac{V(L(t)-)}{\sigma(L(t))}=\frac{V(L(t))}{\sigma(L(t))}\leq \frac{\lbrb{1-\eta}t}{t}=1-\eta\]
	which cannot happen for arbitrary large $t$ on the event $\curly{\limi{t}\frac{V(t)}{\sigma(t)}=1},$ which is of probability one.
\end{proof}

\begin{os}
All subsequent results are stated under the assumption $X_0=0$ a.s.. However, as $X(t)=B(L(t))$ and $B$ is recurrent all these limit results are clearly valid with  $X_0=x$ a.s. for some $x\in\Rb$.
\end{os}
We start with a simple but illuminating example, which covers the case when $\alpha$ takes only two values.
\begin{lem}\label{lem:twoLevels}
	Let $\alpha(x)=\alpha_1,\,x\in A\subseteq \Rb^+$ and $\alpha(x)=\alpha_2,\,x\in\Rb\setminus A$. Let furthermore $\alpha_1,\alpha_2\in \lbrb{0,1}$ and $A$ bounded such that $0<\mathpzc{l}(A), \mathpzc{l}(\partial A)=0.$ Then if $\alpha_2>2\alpha_1$ we have that
	\begin{equation}\label{eq:occupation}
		\limi{t}\frac{\int_{0}^{t}\ind{X(s)\in A}ds}{t}=1,\,\text{a.s.},
	\end{equation}
	and if $\alpha_2<2\alpha_1$ then 
	\begin{equation}\label{eq:occupation1}
	\limi{t}\frac{\int_{0}^{t}\ind{X(s)\in A\cup [-K,K]}ds}{t}=0,\,\text{a.s.}
	\end{equation}
	for any $K\geq 0$.
\end{lem}
\begin{proof}
 First, note that $\sigma_1$ and  $\sigma_2$ defined in \eqref{eq:decom} and \eqref{eq:constr} are respectively stable subordinators of index $\alpha_1$ and  $\alpha_2$ evaluated at an independent of them time $H_t$. Then the conditions \eqref{eq:condi} of Corollary \ref{cor:Bounds1} are satisfied and \eqref{eq:sigma1} together with  Corollary \ref{cor1:growth} translates to 
 \begin{equation}\label{eq:sigma1_1}
 \limi{t}\frac{\sigma_1\lbrb{H_t}}{t^{\frac{1}{2\alpha_1}+\epsilon}}=0;\quad \limi{t}\frac{\sigma_1\lbrb{H_t}}{t^{\frac{1}{2\alpha_1}-\epsilon}}=\infty,\,\,\text{a.s.}
 \end{equation}
 for any $\epsilon>0$ small enough.
 Also, we get in the same fashion that
 \begin{equation}\label{eq:sigma1_2}
 \limi{t}\frac{\sigma_2\lbrb{t-H_t}}{t^{\frac{1}{\alpha_2}+\epsilon}}=0;\quad \limi{t}\frac{\sigma_2\lbrb{t-H_t}}{t^{\frac{1}{\alpha_2}-\epsilon}}=\infty,\,\,\text{a.s.}
 \end{equation}
  If $\alpha_2>2\alpha_1$ then from the second relation of \eqref{eq:sigma1_1} and  the first relation of \eqref{eq:sigma1_2} we get that almost surely
  \begin{equation}\label{eq:comparable}
  	\limi{t}\frac{\sigma_1\lbrb{H_t}}{\sigma(t)}=1,
  \end{equation}
  see \eqref{eq:decom} for the definition of $\sigma$. Recall that in this setting the underlying Markov process in the definition of $X(t-)=M(L(t)-)$, see \eqref{defx}, is the Brownian motion $B$ and $X(t-)=B(L(t))$ from the continuity of $B$. Now since $\mathpzc{l}(\partial A)=0$  we have that 
  \[\Pbb{\exists t\geq 0:\sigma(t)-\sigma(t-)>0; B(t) \in \partial A}=0.\]
 This is true since first almost surely it is true that $\mathpzc{l}\lbrb{\curly{t\geq 0:\, B(t) \in \partial A}}=0$, second almost surely then it holds that
  \[\mathpzc{l}\lbrb{\lbbrbb{0,t}\setminus\curly{t\geq 0:\, B(t) \in \partial A}}=t,\]
 third from \eqref{eq:sigma} and $\sigma$ being a composition of two stable subordinators we have that
  \begin{align}\label{eq:says}
  \mathds{E}\left[ e^{-\lambda \sigma(y)} \mid B(w), w \leq y \right] \, = \, e^{ -\int_0^\infty \l 1-e^{-\lambda s} \r \lbrb{\frac{C_1}{s^{\alpha_1+1}}H_t+\frac{C_2}{s^{\alpha_2+1} }(t-H_t)}ds},
  \end{align}
 and forth relation \eqref{eq:says} implies that since $\curly{t\geq 0:\, B(t) \in \partial A}$ is independent of $\sigma$ and of zero measure then the probability of $\sigma$ jumping at times in this set is zero.
  However, from \eqref{eq:comparable} we have that
   \begin{equation}\label{eq:cond4.6}
   \begin{split}
   &\limi{t}\frac{\sigma_1\lbrb{H_{L(t)}}}{\sigma(L(t))}=1;\quad \limi{t}\frac{\sigma_1\lbrb{H_{L(t)}-}}{\sigma(L(t)-)}=1;  \text{ a.s.}.
   \end{split}
   \end{equation}
 To check the latter note that $\sigma_1(H_t)\leq\sigma(t),t>0$. Assume that there is $\epsilon>0$ and a set $\widetilde{A}$ of positive probability such that on $\widetilde{A}$
  \begin{equation}\label{eq:contrad}
  \begin{split}
  &\limsupi{t}\frac{\sigma_1\lbrb{H_{L(t)}-}}{\sigma(L(t)-)}\leq 1-\epsilon.
  \end{split}
  \end{equation}
  Clearly we can choose $t_0,t_1>0$ and event $\bar{A}$ such that
  \begin{equation*}
  \begin{split}
  &	P(\bar{A})=P\lbrb{\forall t\geq t_0:\frac{\sigma_1\lbrb{H_t}}{\sigma(t)}\geq 1-\frac{\epsilon}{100};t\geq t_1:L(t)\geq t_0+1}\geq 1-\frac12 P(\widetilde{A}).   
  \end{split}
  \end{equation*}
  Then $P(\widetilde{A}\cap \bar{A})>0$ and we work with trajectories in $\widetilde{A}\cap \bar{A}$.  Then for any $\eta\in\lbrb{0,1}$ we have that for any $t>t_1$
  \begin{equation*}
  \begin{split}
  	1-\frac{\epsilon}{100}	    &\leq\frac{\sigma_1\lbrb{H_{L(t)-\eta}}}{\sigma(L(t)-\eta)}\leq  \frac{\sigma_1\lbrb{H_{L(t)-}}}{\sigma(L(t)-\eta)}\\
  &= \frac{\sigma_1\lbrb{H_{L(t)-}}}{\sigma(L(t)-)}\frac{\sigma(L(t)-)}{\sigma(L(t)-\eta)}\leq \lbrb{1-\epsilon}\frac{\sigma(L(t)-)}{\sigma(L(t)-\eta)}.
  \end{split}
  \end{equation*}
  Setting $\eta\to 0$ and using that $\limo{\eta}\sigma(L(t)-\eta)=\sigma(L(t)-)$ we arrive at a contradiction.
  Next, let 
  \[A_t=\curly{\Delta \sigma(L(t))=\Delta \sigma_1(H_{L(t)})}.\]
  Then, clearly
  \[\mathpzc{l}\lbrb{s\leq t:X(s)\in A}=\sigma_1(H_{L(t)}-)+\lbrb{t-\sigma(L(t)-)}\ind{A_t}\]
 and on $A_t$ 
\[\mathpzc{l}\lbrb{s\leq t:X(s)\in A}=t-\sigma_2\lbrb{\lbrb{t-H_{L(t)}}-}.\]  
  Fix $\eta\in\lbrb{0,1}$ and set
  \[A_{t,\eta}=\curly{\sigma_1(H_{L(t)}-)\leq\lbrb{1-\eta}t}.\]
Since  $\sigma(L(t))=\sigma_1\lbrb{H_{L(t)}}+\sigma_2\lbrb{L(t)-H_{L(t)}}$ and \eqref{eq:cond4.6} holds true then the conditions of Proposition \ref{prop:Close} are satisfied in the almost surely sense  and we have from \eqref{eq:Close11} that almost surely there exists $t_0$ depending on the path and $\eta$ such that $A_{t,\eta}\cap A^c_t=\emptyset$ for all $t\geq t_0$. This leads to
  \begin{equation*}
  \begin{split}
  &\limi{t}\frac{\mathpzc{l}\lbrb{s\leq t:X(s)\in A}}{t}\\
  &=\limi{t}\lbrb{\frac{\mathpzc{l}\lbrb{s\leq t:X(s)\in A}\ind{A_{t,\eta}}}{t}+\frac{\mathpzc{l}\lbrb{s\leq t:X(s)\in A}\ind{A^c_{t,\eta}}}{t}}\\
  &\geq \limi{t}\lbrb{\frac{\lbrb{\sigma_1(H_{L(t)}-)+\lbrb{t-\sigma(L(t)-)}\ind{A_t}}\ind{A_{t,\eta}}}{t}+\frac{\sigma_1(H_{L(t)}-)\ind{A^c_{t,\eta}}}{t}}\\
  &\geq \limi{t}\lbrb{\frac{\lbrb{\sigma_1(H_{L(t)}-)+\lbrb{t-\sigma(L(t)-)}\ind{A_t}}\ind{A_{t,\eta}}}{t}+\lbrb{1-\eta}\ind{A^c_{t,\eta}}}\\
  &=\limi{t}\lbrb{\frac{t-\sigma_2\lbrb{\lbrb{t-H_{L(t)}}-}}{t}\ind{A_{t,\eta}}+\lbrb{1-\eta}\ind{A^c_{t,\eta}}}\\
  &\geq\limi{t}\lbrb{1-\eta}\ind{A_{t,\eta}}+\lbrb{1-\eta}\ind{A^c_{t,\eta}}=\lbrb{1-\eta},
  \end{split}
  \end{equation*}
  where the very last inequality follows from fact that a.s.
  \[1=\limi{t}\frac{\sigma_1(H_{L(t)}-)}{\sigma(L(t)-)}=\limi{t}\frac{\sigma_1(H_{L(t)}-)}{\sigma_1(H_{L(t)}-)+\sigma_2\lbrb{\lbrb{t-H_{L(t)}}-}}\]
  and hence \[\limi{t}\frac{\sigma_2\lbrb{\lbrb{t-H_{L(t)}}-}}t\leq \limi{t}\frac{\sigma_2\lbrb{\lbrb{t-H_{L(t)}}-}}{\sigma_1(H_{L(t)}-)}=0.\]
  Since $\eta$ is arbitrary we get that
  \[\limi{t}\frac{\mathpzc{l}\lbrb{s\leq t:X(s)\in A}}{t}=1,\]
  which proves \eqref{eq:occupation}. 
  
  Let next $2\alpha_1>\alpha_2$. Then using exactly the same arguments and \eqref{eq:sigma1_2} we arrive at
   \begin{equation}\label{eq:comparable11}
  \limi{t}\frac{\sigma_2\lbrb{t-H_t}}{\sigma(t)}=1
  \end{equation}
  and almost surely
  \[\limi{t}\frac{\mathpzc{l}\lbrb{s\leq t:X(s)\in A^c}}{t}=1.\]
  However, if, for some $K>0$, we set $H_t(1), H_t(2)$ the occupation measures of the Brownian motion  of $\lbbrbb{0,K}\setminus A$ and $\lbbrb{-K,0}$, then 
  \[\sigma_2\lbrb{t-H_t}=\sigma_2\lbrb{t-H_t-H_t(1)- H_t(2)}+\sigma_2\lbrb{H_t(1)}+\sigma_2\lbrb{H_t(2)}.\]
  Clearly, the same reasoning for $H_t$ applies to $H_t(1), H_t(2)$ and yields through utilization of Corollaries \ref{cor:Bounds}, \ref{cor:Bounds1} to 
  \begin{equation}\label{eq:sigma1_22}
  \limi{t}\frac{\sigma_2\lbrb{H_t(i)}}{t^{\frac{1}{2\alpha_2}+\epsilon}}=0;\quad \limi{t}\frac{\sigma_2\lbrb{H_t(i)}}{t^{\frac{1}{2\alpha_2}-\epsilon}}=\infty,\,\,\text{a.s.}, i=1,2.
  \end{equation}
  This  allows us to deduct that \eqref{eq:comparable11} is further augmented to
  \begin{equation*}
  \limi{t}\frac{\sigma_2\lbrb{t-H_t-H_t(1)-H_t(2)}}{\sigma(t)}=1,\,\text{a.s.}.
  \end{equation*}
  As before we can again deduct that
  \[\limi{t}\frac{\mathpzc{l}\lbrb{s\leq t:X(s)\in A^c\cup (-\infty, -K)\cup (K,\infty)}}{t}=1.\]
  This concludes the proof.
\end{proof}

Let us now assume that $\alpha:\Rb\mapsto\lbrb{0,1}$, $\sup_{x \in \Rb}\alpha(x)<1$ and set $\alpha^*=\min_{x\in\Rb}\alpha(x)>0$. Assume further that there exists $\beta$ small enough such that $A_\beta=\curly{x\in\Rb:\alpha(x)<\alpha^*+\beta<1}$ is bounded and satisfies $0<\mathpzc{l}\lbrb{A_\beta}<\infty$ and for any $0<\epsilon<\beta, \mathpzc{l}\lbrb{A_\epsilon}>0$. Also let $\mathpzc{l}\lbrb{\partial A_\epsilon}=0$ for all $\epsilon\leq\beta$. Without loss of generality let $A_\beta\subseteq\Rb^+$. Denote by $H_s(\beta)$ the occupation measure of $A_\beta$ as above and $\sigma_1(H_s(\beta))$ the process in the decomposition \eqref{eq:decom}. Then we have the result.
\begin{lem}\label{lem:GenBounds}
	With the conditions on $A_\beta$ and $\alpha$ above we have that for any $\beta>\varepsilon>0$ small enough
	\begin{equation}\label{eq:GenBounds1}
		\limi{t}\frac{\sigma_1(H_s(\beta))}{\lbrb{H_s(\beta)}^{\frac{1}{\alpha^*}+\varepsilon}}=0
	\end{equation}
	and 
	\begin{equation}\label{eq:GenBounds2}
	\limi{t}\frac{\sigma_1(H_s(\beta))}{\lbrb{H_s(\beta)}^{\frac{1}{\alpha^*+\varepsilon}-\varepsilon}}=\infty.
	\end{equation}
\end{lem}
\begin{proof} 
	Recall that the intensity measure of $\sigma$ is in general
	\begin{equation}\label{eq:intensity}
		\nu\lbrb{ds,x}=\frac{\alpha(x)}{\Gamma\lbrb{1-\alpha(x)}}\frac{ds}{s^{\alpha(x)+1}}= v_x(s)ds,\,s\in\Rb^+,\,x\in \Rb.
	\end{equation}
	Set
	\begin{equation}\label{eq:cond}
	0<q_1:=\inf_{x\in A_\beta}\frac{\alpha(x)}{\Gamma(1-\alpha(x))}\leq \sup_{x\in A_\beta}\frac{\alpha(x)}{\Gamma(1-\alpha(x))}:=q_2<\infty.
	\end{equation}
	From \eqref{eq:intensity} the density of the intensity measure of $\sigma$ can be estimated uniformly on $\lbrb{s,x}\in\Rb^+\times \Rb$ as 
		\begin{equation}\label{eq:b1}
		v_x(s)\leq q_2 \lbrb{s^{-\alpha^*-\beta}+s^{-\alpha^*}}
		\end{equation}
	and for $\varepsilon<\beta$ on $\lbrb{s,x}\in\Rb^+\times A_\varepsilon$ (recall that $A_\varepsilon=\curly{x\in\Rb:\alpha(x)<\alpha^*+\varepsilon<1}$)
	\begin{equation}\label{eq:b2}v_x(s)\geq q_1 s^{-\alpha^*-\varepsilon}\ind{s\geq 1}.
	\end{equation}
	Therefore, if $\sigma^\eta$ is a stable subordinator of index $\eta$ assume that we can construct pathwise stable subordinators of index $\alpha^*,\alpha^*+\varepsilon,\alpha^*+\beta$ such that
	\begin{equation}\label{eq:stochBounds}
		\begin{split}
		&\sigma_1(H_s(\beta))\leq \sigma^{\alpha^{*}}(c_1H_s\lbrb{\beta})+\sigma^{\alpha^{*}+\beta}(c_1H_s\lbrb{\beta}),\\
		&\sigma_1(H_s(\beta))\geq \sigma^{\alpha^{*}+\varepsilon}(c_2H_s\lbrb{\beta})-\sum_{v\leq c_2H_s\lbrb{\beta}}\Delta \sigma^{\alpha^{*}+\varepsilon}_v\ind{\Delta \sigma^{\alpha^{*}+\varepsilon}_v\leq 1}
	    \end{split}
	\end{equation}
	where $c_1, c_2>0$ and in the second inequality we have truncated the jumps less or equal to $1$.
	However, from an easy application of Corollary \ref{cor1:growth} to both $\sigma^{\alpha^{*}}, \sigma^{\alpha^{*}+\beta}$ at infinity we get almost surely that for any $\epsilon>0$ as small as we wish \begin{equation}\label{eq:stochBounds4}
	\begin{split}
	&\liminfi{s}\frac{\sigma^{\alpha^{*}}(c_1H_s\lbrb{\beta})}{\lbrb{c_1H_s(\beta)}^{\frac{1}{\alpha^*}-\epsilon}}=\infty;\quad \limsupi{s}\frac{\sigma^{\alpha^{*}+\beta}(c_1H_s\lbrb{\beta})}{\lbrb{c_1H_s(\beta)}^{\frac{1}{\alpha^*+\beta}+\epsilon}}=0.
	\end{split}
	\end{equation}
	Henceforth as long as $\frac1{\alpha^*+\beta}+2\epsilon<\frac{1}{\alpha^*}$ then $\sigma^{\alpha^{*}}\lbrb{c_1H_s\lbrb{\beta}}$ dominates the right-hand side of the first inequality in \eqref{eq:stochBounds}.
	Also $\sigma^{\alpha^{*}+\varepsilon}(c_2H_s\lbrb{\beta})$ dominates the second term in the right-hand side of the second inequality in \eqref{eq:stochBounds} since the small jumps have all exponential moments and have slow growth. Therefore, almost surely,
	\begin{equation*}
	\begin{split}
	&\limsupi{s}\frac{\sigma_1(H_s(\beta)) }{\sigma^{\alpha^{*}}(c_1H_s\lbrb{\beta})}	\leq 1	  \text{ and }   \limsupi{s}\frac{\sigma_1(H_s(\beta)) }{\sigma^{\alpha^{*}+\varepsilon}(c_2H_s\lbrb{\beta})}	\geq 1.
	\end{split}
	\end{equation*}
	 Thus, \eqref{eq:GenBounds1} and \eqref{eq:GenBounds2} follow respectively from  another application of Corollary \ref{cor1:growth} in the last relations. This proves the claims modulo to verification of the pathwise construction in  \eqref{eq:stochBounds}. Let us start with the upper bound therein.	
It can be obtained by the addition of two independent processes with density of the intensity measures of the form
\begin{align*}
&\widetilde{v}_x(s):=\lbrb{\lbrb{hs^{-\alpha^*-1}-\frac{\alpha(x)s^{-\alpha(x)-1}}{\Gamma(1-\alpha(x))}}\ind{s> 1}+hs^{-\alpha^*-1}\ind{s\leq 1}}\ind{x\in A_\beta}\\
&\bar{v}_x(s):=\lbrb{\lbrb{hs^{-\alpha^*-\beta-1}-\frac{\alpha(x)s^{-\alpha(x)-1}}{\Gamma(1-\alpha(x))}}\ind{s\leq 1}+hs^{-\alpha^*-\beta-1}\ind{s> 1}}\ind{x\in A_\beta}   
\end{align*}	
	as long as $h>q_2=\sup_{x\in A_{\beta}} \frac{\alpha(x)}{\Gamma(1-\alpha(x))}$ which ensures the positivity of $\widetilde{v}_x,\bar{v}_x$ on $\Rb^+\times A_\beta$. Then on $A_\beta$ the total intensity of the sum of the three independent  processes is
\begin{equation*}
\begin{split}
&v_x(s)+\widetilde{v}_x(s)+	\bar{v}_x(s)=h s^{-\alpha^*-1}+hs^{-\alpha^*-\beta-1}	    
\end{split}
\end{equation*}	
or the process is also the sum of two independent copies of time changed  stable subordinators $\sigma^{\alpha^{*}}(c_1\cdot), \sigma^{\alpha^{*}+\beta}(c_2\cdot)$. We simply recall that if $\chi$ is a subordinator with \LL measure $\Pi$ then the \LL measure of $\chi_{ct}$ is $c\Pi$ and by choosing $c_1=\frac{h\Gamma\lbrb{1-\alpha^*}}{\alpha^*}$ and $c_2=\frac{h\Gamma\lbrb{1-\alpha^*-\beta}}{\alpha^*+\beta}$ we ensure that the intensity is respectively $h s^{-\alpha^*-1}$ and $h s^{-\alpha^*-\beta-1}$.
 The lower bound in \eqref{eq:stochBounds} can be obtained by thinning of the jumps of $\sigma_1$, say $\Delta=\lbrb{\Delta_s}_{s\geq 0}$, in the manner
	\[\sum_{s}\Delta_s \epsilon_s \]
	with $\epsilon_s$ independent of $\sigma_1$ Bernoulli random variable with parameter
	\[p(B_s,\Delta_s)=\frac{h \Delta_s^{-\alpha^*-\varepsilon -1}}{\frac{\alpha(B_s)}{\Gamma(1-\alpha(B_s))}\Delta_s^{-\alpha(B_s)-1}}\ind{\abs{\Delta_s}>1}.\]
	This procedure thins the jumps accordingly as long as 
	\[\frac{h}{\frac{\alpha(B_s)}{\Gamma(1-\alpha(B_s))}}<1,\]
	the choice of which is always possible on $B_s\in A_\varepsilon$ from \eqref{eq:cond} and ensures that
	\[\frac{h \Delta_s^{-\alpha^*-\varepsilon -1}}{\frac{\alpha(B_s)}{\Gamma(1-\alpha(B_s))}\Delta_s^{-\alpha(B_s)-1}}<1.\]
	Then the intensity of the thinned process is given for $x\in A_{\varepsilon}$ by
	\[\widetilde{v}_x(s)=\frac{h s^{-\alpha^*-\varepsilon -1}}{\frac{\alpha(x)}{\Gamma(1-\alpha(x))}s^{-\alpha(x)-1}}v_x(s)\ind{s\geq 1}=h s^{-\alpha^*-\varepsilon -1}\ind{s\geq 1}\]
	or that of a time changed stable process whose jumps smaller than $1$ have been trimmed away.
\end{proof} 
However, \eqref{eq:GenBounds1} and \eqref{eq:GenBounds2} can be combined with \eqref{eq:Bounds1} to yield the following almost sure estimates on the growth of $\sigma_1(H_s(\beta))$.
\begin{coro}\label{cor:stochBounds} For all $\varepsilon>0$ small enough
	\begin{equation}\label{eq:GenBounds11}
	\limi{t}\frac{\sigma_1(H_t(\beta))}{t^{\frac{1}{2\alpha^*}+\varepsilon}}=0
	\end{equation}
	and 
	\begin{equation}\label{eq:GenBounds21}
	\limi{t}\frac{\sigma_1(H_t(\beta))}{\lbrb{t^{\frac{1}{2\alpha^*+2\varepsilon}-\varepsilon}}}=\infty.
	\end{equation}
\end{coro}
Next, we estimate the growth of $\sigma_2$.
\begin{coro}\label{cor:stochBounds1}
	Let $\alpha,A_\beta$ be as in Lemma \ref{lem:GenBounds}. Assume in addition that $\alpha^*=\min_{x\in\Rb}\alpha(x)>0, \max_{x\in\Rb}\alpha(x)<1$, $1>\limi{x}\alpha(x)=\alpha_I>\alpha^*$ and $1>\lim_{x\to -\infty}\alpha(x)=\alpha_J>\alpha^*$. Then, with $\alpha_\diamond=\min\curly{\alpha_I,\alpha_J,2\lbrb{\alpha^*+\beta}}$  
	\begin{equation}\label{eq:GenBounds11_1}
     	\limi{t}\frac{\sigma_2\lbrb{t-H_t\lbrb{\beta}}}{t^{\frac{1}{\alpha_\diamond}+\epsilon}}=0;\,\, \limi{t}\frac{\sigma_2\lbrb{t-H_t\lbrb{\beta}}}{t^{\frac{1}{\alpha_\diamond}-\epsilon}}=\infty,\text{ a.s.,}
	\end{equation}
	for all $\epsilon$ small enough.
\end{coro}
\begin{proof}
	Fix $\varepsilon>0$. Let $K>0$ be large enough so that $\sup_{x< -K}\abs{\alpha(x)-\alpha_J}\leq \varepsilon/100$ and $\sup_{x>K}\abs{\alpha(x)-\alpha_I}\leq \varepsilon/100$. Let \[A_\circ=A_\beta\cup[-K,K],\] where clearly $\mathpzc{l}\lbrb{A_\circ}\in\lbrb{0,\infty}$. Also set $A_\diamond=A_\circ\setminus A_\beta$. Then the occupation measure of $A_\circ$ is evaluated as 
	\[H^\circ_t=H_t\lbrb{\beta}+H^\diamond_t\]
	and from the construction of $\sigma_1$ and $\mathpzc{l}\lbrb{\partial A_\beta}=0$, we check that $\sigma_2$ grows only on $A_\diamond$ 	or by the increase of $H^\diamond$. Since $\alpha(x)>\alpha_*+\beta$ on $A^c_\beta$ and $\mathpzc{l}\lbrb{A_\circ}\in\lbrb{0,\infty}$ absolutely the same arguments as in Lemma \ref{lem:GenBounds} and Corollary \ref{cor:stochBounds} yield that
	\begin{equation}\label{eq:stochBounds1_1}
	\begin{split}
	& \limi{t}\frac{\sigma_2(H^\diamond_t)}{t^{\frac{1}{2\lbrb{\alpha^*+\beta}}+\eta_1}}=0;\,\,\limi{t}\frac{\sigma_2(H^\diamond_t)}{t^{\frac{1}{2\lbrb{\alpha^*+\beta}}-\eta_1}}=\infty\text{ a.s.,}
	\end{split}
	\end{equation}
	for all $\eta_1>0$ small enough.  Next, note that 
	\begin{equation}\label{eq:decomp}
	\sigma_2\lbrb{t-H_t\lbrb{\beta}}=\sigma_2\lbrb{t-H^\circ_t}+\sigma_2\lbrb{H^\diamond_t}.
	\end{equation}
    Precisely, as the construction leading to \eqref{eq:stochBounds}, denoting $\alpha_\circ=\min\curly{\alpha_I,\alpha_J}$,  we can show that
    \begin{equation}\label{eq:stochBounds2}
    \begin{split}
    &\sigma_2(t-H^\circ_t)\leq \sigma^{\alpha_{\circ}}(c_1\lbrb{t-H^\circ_t})+\sigma^{\alpha_{\circ}+\eta}(c_1\lbrb{t-H^\circ_t}),\\
    &\sigma_2(\lbrb{t-H^\circ_t})\geq \sigma^{\alpha_{\circ}+\eta}(c_2\lbrb{t-H^\circ_t})-\sum_{v\leq c_2\lbrb{t-H^\circ_t}}\Delta \sigma^{\alpha_{\circ}+\eta}_v\ind{\Delta \sigma^{\alpha_{\circ}+\eta}_v\leq 1},
    \end{split}
    \end{equation}
    where $\varepsilon/100<\eta<\varepsilon/2$ and $c_1,c_2$ correspond to time changes related to estimates of the densities precisely as in \eqref{eq:b1} and \eqref{eq:b2}, and $\sigma^{\cdot}$ stands for stable subordinator of index $\cdot\in\lbrb{0,1}$. However, as in \eqref{eq:Bounds1} we have that $H^\circ_t$ grows almost surely sublinearly and thus we can conclude that for any such $\eta$  small enough
    	\begin{equation*}
    \begin{split}
    &\liminfi{t}\frac{\sigma_2(t-H^\circ_t)}{t^{\frac{1}{\alpha_{\circ}}-\eta}}=\infty;\quad \limsupi{t}\frac{\sigma_2(t-H^\circ_t)}{t^{\frac{1}{\alpha_{\circ}+\eta}+\eta}}=0.
    \end{split}
    \end{equation*}
	Since $K$ can be chosen as large as we wish and thus $\varepsilon$ and $\eta$ as small as we wish, we deduct  via \eqref{eq:decomp} and \eqref{eq:stochBounds1_1}  the validity of \eqref{eq:GenBounds11_1} for all $\epsilon$ small enough.  This settles the proof of the corollary.
\end{proof}
Then the following result holds true
\begin{te}\label{thm:occupation}
	Let  $\alpha:\Rb\mapsto\lbrb{0,1}$, $\alpha^*=\min_{x\in\Rb}\alpha(x)>0, \max_{x\in\Rb}\alpha(x)<1$ and \[1>\limi{x}\alpha(x)=\alpha_I>\alpha^*,\,\,1>\lim_{x\to -\infty}\alpha(x)=\alpha_J>\alpha^*.\] Also let there exist $\beta_0$ small enough such that for all $\beta_0\geq \beta$, the set $A_\beta=\curly{x\in\Rb:\alpha(x)<\alpha^*+\beta<1}$ is bounded and satisfies $0<\mathpzc{l}\lbrb{A_\beta}<\infty$ and also $\mathpzc{l}\lbrb{\partial A_\beta}=0$. Then,
	\begin{enumerate}
		\item if $2\alpha^*<\min\curly{\alpha_I,\alpha_J}$ we have that for any $\beta\leq \beta_0$
			\begin{equation}\label{eq:occupationG}
		\limi{t}\frac{\int_{0}^{t}\ind{X(s)\in A_\beta}ds}{t}=1,\,\text{a.s.};
		\end{equation}
		\item  and if $2\alpha^*>\min\curly{\alpha_I,\alpha_J}$, for any $K>0$,
		\begin{equation}\label{eq:occupationG1}
		\limi{t}\frac{\int_{0}^{t}\ind{X(s)\in A^c_\beta\cap \lbbrbb{-K,K}^c}ds}{t}=1,\,\text{a.s.}.
		\end{equation}
	\end{enumerate}
	
\end{te}
\begin{proof} Recall that $X(t)=B(L(t))$.
	Let $2\alpha^*<\min\curly{\alpha_I,\alpha_J}$ and choose $\beta'>0$ small enough so that even $2\alpha^*+2\beta'<\min\curly{\alpha_I,\alpha_J}$. Choose $\beta<\beta'$ so that the conditions of the theorem are satisfied. Then using Corollaries \ref{cor:stochBounds}, \ref{cor:stochBounds1} we get precisely as in the proof of  Lemma \ref{lem:twoLevels} that the equivalent to \eqref{eq:comparable} relation holds, that is
	 \begin{equation}\label{eq:comparable1}
	\limi{t}\frac{\sigma_1\lbrb{H_t\lbrb{\beta}}}{\sigma(t)}=1,\text{ a.s.}.
	\end{equation}
	Since again
	\[\mathpzc{l}\lbrb{\lbbrbb{0,t}\setminus\curly{t\geq 0:\, B(t) \in \partial A_\beta}}=t\]
	and the validity of Proposition \ref{prop:Close} in the almost sure sense is at hand, precisely as in the proof of \eqref{eq:occupation} of Lemma \ref{lem:twoLevels} we establish \eqref{eq:occupationG}. Assume next that $2\alpha^*>\min\curly{\alpha_I,\alpha_J}$. Then for any $\beta>0$, \[\alpha_\diamond=\min\curly{\alpha_I,\alpha_J,2\lbrb{\alpha^*+\beta}}=\min\curly{\alpha_I,\alpha_J}\]  
	and from \eqref{eq:GenBounds21} and \eqref{eq:GenBounds11_1}  we conclude that
	\[\limi{t}\frac{\sigma_2(t-H_t)}{\sigma(t)}=1\text{ a.s..}\]
	Moreover, from \eqref
	{eq:stochBounds1_1} and \eqref{eq:decomp} and with \[H^\diamond_t=\int_{0}^{t}\ind{B(s)\in A^\diamond}ds;\,\,H^\circ_t=\int_{0}^{t}\ind{B(s)\in A^\circ}ds,\]
	 where for any $K>0$,   $A_\circ=A_\beta\cup[-K,K]$ and $A_\diamond=A_\circ\setminus A_\beta$, we have that
	\[\limi{t}\frac{\sigma_2(t-H^\circ_t)}{\sigma(t)}=1\text{ a.s..}\]
	Then the proof follows precisely as the proof of case $2\alpha_1>\alpha_2$  of Lemma \ref{lem:twoLevels}.
\end{proof}

When $A_0=\curly{x\in\Rb:\alpha(x)=\alpha^*}$ is a  bounded disjoint union of intervals which implies that $\mathpzc{l}\lbrb{A_0}\in\lbrb{0,\infty}$ and for all small $\beta>0$,  $A_0=A_\beta$, where $A_\beta=\curly{x\in\Rb:\alpha(x)<\alpha^*+\beta}$, then we have the stronger result which localizes  in probability the anomalous diffusion.
\begin{te}\label{thm:hyp}
	Let  $\alpha:\Rb\mapsto\lbrb{0,1}$ and $A_0=\curly{x\in\Rb:\alpha(x)=\alpha^*}=\bigcup_i I_i,$ be bounded and where $I_i$ are disjoint intervals. Let also $\mathpzc{l}\lbrb{A_0}\in\lbrb{0,\infty}$, $\mathpzc{l}\lbrb{\partial A_0}=0$ and for all small $\beta>0$,  $A_0=A_\beta$, where $A_\beta=\curly{x\in\Rb:\alpha(x)<\alpha^*+\beta}$. Finally, let $\alpha^*=\min_{x\in\Rb}\alpha(x)>0, \max_{x\in\Rb}\alpha(x)<1$ and \[1>\limi{x}\alpha(x)=\alpha_I,\,\,1>\lim_{x\to -\infty}\alpha(x)=\alpha_J.\] Then if $2\alpha^*<\min\curly{\alpha_I,\alpha_J}$ it holds true that 
	\begin{equation}\label{eq:hyp}
		\limi{t}\Pbb{X(t)\in A_0}=1.
	\end{equation}
\end{te}
For clarity let us consider a special case which is of greatest interest.
\begin{coro}\label{cor:conseq}
	 Let   $\alpha:\Rb\mapsto\lbrb{0,1}$ be piece-wise constant taking values $0<\alpha_1<\alpha_2<\cdots<\alpha_n<1$. Let $A=\curly{x\in\Rb:\alpha(x)=\alpha_1}$ be a finite union of intervals such that $\mathpzc{l}\lbrb{A}\in\lbrb{0,\infty}$ and $\min\curly{\limi{x}\alpha(x);\lim_{x\to -\infty}\alpha(x)}=\alpha_j, 2\leq j\leq n$. If $2\alpha_1<\alpha_j$ then 
	 	\begin{equation}\label{eq:hyp1}
	 \limi{t}\Pbb{X(t)\in A}=1.
	 \end{equation}
	 Otherwise, if $2\alpha_1>\alpha_j$ and $A_j=\curly{x\in \Rb:\alpha(x)=\alpha_j}$, then
	 	\begin{equation}\label{eq:hyp2}
	 \limi{t}\Pbb{X(t)\in A_j}=1.
	 \end{equation}
\end{coro}
We proceed with the proof of the Theorem \ref{thm:hyp}.
\begin{proof}[Proof of Theorem \ref{thm:hyp}]
	If $2\alpha^*<\min\curly{\alpha_I,\alpha_J}$ we choose $\beta>0$ small enough that $A=A_\beta$. From \eqref{eq:comparable1} we get that
	\begin{equation}\label{eq:comparable2}
	\limi{t}\frac{\sigma_1\lbrb{H_t}}{\sigma(t)}=1,\text{ a.s.},
	\end{equation}
	where we recall that $X(t)=B(L(t))$, $L(t)=\inf\curly{s>0:\sigma(s)>t}$ and from \eqref{eq:decom} and the assumptions of the theorem
	\[\sigma(t)=\sigma_1\lbrb{H_t}+\sigma_2\lbrb{t-H_t}\]
	with $H_t=\int_{0}^{t}\ind{B(s)\in A}ds$. From \eqref{eq:comparable2} we arrive at
	\[	\limi{t}\frac{\sigma_1\lbrb{H_{L(t)}}}{\sigma(L(t))}=1,\text{ a.s.}.\]
	Therefore \eqref{eq:Close11} of Proposition \ref{prop:Close} is valid and hence for any $\eta\in\lbrb{0,1}$
	\[\limi{t}\Pbb{\sigma_1\lbrb{H_{L(t)}}\leq\lbrb{1-\eta}t; \Delta \sigma_2(L(t)-H_{L(t)})=\Delta \sigma(L(t))}=0\]
	where we have used that $P(\Delta \sigma(L(t))=0)=0$.
	Henceforth, if \eqref{eq:hyp1} fails then for any $\eta\in\lbrb{0,1}$ and some $c\in\lbrb{0,1}$
	\begin{equation}\label{eq:contra}
	\begin{split}
	c&\leq \limsupi{t}\Pbb{X(t)\notin A}\\
	&=\limsupi{t}\Pbb{\sigma_1\lbrb{H_{L(t)}}\in\lbrb{\lbrb{1-\eta}t,t}; \Delta \sigma_2(L(t)-H_{L(t)})=\Delta \sigma(L(t))}\\
	&\leq \limsupi{t}\Pbb{\sigma_1\lbrb{L_1(t)-}\in\lbrb{\lbrb{1-\eta}t,t}},
	\end{split}
	\end{equation}
	where the very last inequality follows easily from $\sigma_1\lbrb{H_{L(t)}}\in\lbrb{\lbrb{1-\eta}t,t}$ and $\sigma_1\lbrb{H_{L(t)}}\leq \sigma(L_1(t)-)< t$ with $L_1(t)=\inf\curly{s>0:\sigma_1(s)>t}$. Next, note that $\sigma_1$ is a stable subordinator of index $\alpha_1$, see the definition of $A$, and \eqref{eq:intensity} is the form of its intensity measure. Also from \cite[Chapter III, Section 1]{bertoinb} we know that the potential density of $\sigma_1$ is 
	\[u_1(x)=Cx^{\alpha_1-1}, x>0,\]
	and from \cite[Chapter III, Proposition 2]{bertoinb}
	\begin{equation*}
	\begin{split}
	&\Pbb{\sigma_1\lbrb{L_1(t)-}\in\lbrb{\lbrb{1-\eta}t,t}}=C\int_{(1-\eta)t}^{t}\bar{\nu}(t-y)y^{\alpha_1-1}dy\\
	&=D\int_{(1-\eta)t}^{t}\lbrb{t-y}^{-\alpha_1}y^{\alpha_1-1}dy=D\int_{(1-\eta)}^{1}\lbrb{1-y}^{-\alpha_1}y^{\alpha_1-1}dy,
	\end{split}
	\end{equation*}
	where $D>0$ is the multiplication of the constant of the potential density and \eqref{eq:intensity}.  However, for any $\eta$ small we then get that 
	\[\limi{t}\Pbb{\sigma_1\lbrb{L_1(t)-}\in\lbrb{\lbrb{1-\eta}t,t}}<\frac{c}{2},\]
	which contradicts \eqref{eq:contra}. Therefore, we conclude that \eqref{eq:hyp1} holds true.
\end{proof}
\begin{proof}[Proof of Corollary \ref{cor:conseq}]
	Relation \eqref{eq:hyp1} is an immediate consequence of Theorem \ref{thm:hyp}. The proof of \eqref{eq:hyp2} in fact carries on using the same arguments as in the proof of Theorem  \ref{thm:hyp}. We summarize them as follows:
	\begin{itemize}
		\item since $2\alpha_1>\alpha_j$ then from Corollary \ref{cor:stochBounds} and Corollary \ref{cor:stochBounds1} we deduct that $\limi{t}\sigma_2(t-H^\circ_t)/\sigma(t)=1$ a.s., where $A_\circ=A\cup [-K,K]$;
		\item on $A_j$ the subordinator $\sigma_2$ is stable of index $\alpha_j$;
		\item therefore \eqref{eq:Close11} of Proposition \ref{prop:Close} is valid and contradiction with it is established, provided $\liminfi{t}\Pbb{X(t)\in A_j}<1$, precisely as in the proof of Theorem \ref{thm:hyp}.
	\end{itemize}
This completes the proof.
\end{proof}
\subsection{Unbounded set}
We consider now the situation \[\limi{x}x^{-c}\mathpzc{l}\lbrb{A\cap\lbbrbb{-x,x}}=a\in\lbrb{0,\infty},c\in\lbbrb{0,1}\] and $A$ unbounded.
Consider $A_1=A\cap\lbbrb{0,\infty}, A_2=A\cap \lbrb{-\infty,0}$. Assume further that
\begin{equation}\label{eq:asymp}
\begin{split}
&\limi{x}x^{-{c_1}}\mathpzc{l}\lbrb{A_1\cap\lbbrbb{0,x}}=a_1\in\lbrb{0,\infty}\\
&\limi{x}x^{-{c_2}}\mathpzc{l}\lbrb{A_2\cap\lbbrbb{-x,0}}=a_2\in\lbrb{0,\infty}
\end{split}
\end{equation}
and without loss of generality that $1> c_1\geq c_2\geq 0$. In this case, we say that $A$ satisfies the growth Assumption $(G)$. Set
\begin{equation}\label{eq:H1}
H_t(A)=\int_{0}^{t}\lbrb{\ind{B_s\in A_1}+\ind{B_s\in A_2}}ds=H^1_t+H^2_t.
\end{equation}
Also we introduce 
\begin{equation}\label{eq:G1}
F_1(t):=\int_{0}^{t}\ind{A_1\cap\lbbrbb{0,x}}dx;\quad F_2(t):=\int_{0}^{t}\ind{A_2\cap\lbbrbb{-x,0}}dx
\end{equation}
and from \eqref{eq:asymp} we have that
\begin{equation}\label{eq:G2}
F_i(t)\sim a_it^{c_i} \text{ $i=1,2$}.
\end{equation}
From \cite[Proposition 9.5]{bertoins} we arrive at 
the following fact 
\begin{lem}\label{lem:inf}
	The subordinators $H^i_{\tau(t)},i=1,2,$ have Laplace exponents $\Phi_i, i=1,2,$ that are regularly varying and of the type
	\begin{equation}\label{eq:inf}
		\Phi_i\lbrb{\lambda}\simo b_i \lambda^{\frac{1}{1+c_i}},\text{ $i=1,2,$}
	\end{equation}
	where $b_i$ are some positive and finite constants. Moreover, for any $\epsilon>0$ small enough,
\begin{equation}\label{eq:inf1}
\begin{split}
&\liminfi{t}\frac{H^i_t}{t^{\frac{1+c_i-\epsilon}{(2+\epsilon)}}}=\infty\text{ and }\limsupi{t}\frac{H^i_t}{t^{\frac{1+c_i
			+\epsilon}{(2-\epsilon)}}}=0\text{ a.s.}
\end{split}
\end{equation}
\end{lem}
\begin{proof}
	The proof of \eqref{eq:inf} is immediate from substitution in the quantities involved in the statement of \cite[Proposition 9.5]{bertoins}. Relation \eqref{eq:inf1} follows from the fact that $\Phi_i$ are regularly varying which leads to \eqref{eq:Bounds} of Proposition \ref{prop:Bounds} and the fact  \eqref{eq:inf} ensures that  Theorem \ref{thm:growth} holds true and yield for any $\epsilon>0$ small enough 
	\begin{equation}\label{eq:growth}
	\begin{split}
	\liminfi{t}\frac{H_{\tau\lbrb{t}}^i}{t^{1+c_i-\epsilon}}=\infty\text{ and } \limsupi{t}\frac{H_{\tau\lbrb{t}}^i}{t^{1+c_i+\epsilon}}=0.
	\end{split}
	\end{equation} 
\end{proof}
Let us now assume that $\alpha:\Rb\mapsto\lbrb{0,1}$ and set $\alpha^*=\min_{x\in\Rb}\alpha(x)$. Assume further that $A=\curly{x\in\Rb:\alpha(x)=\alpha^*<1}$ satisfies the growth Assumption (G) with $1>c_1\geq c_2\geq 0$, see \eqref{eq:asymp}. Also let us suppose that $\mathpzc{l}\lbrb{\partial A}=0$ and $A=A_\beta=\curly{x\in\Rb:\alpha(x)<\alpha^*+\beta<1}$ for all $\beta>0$ small enough. The occupation measure of $A$ is as in \eqref{eq:H1} as above. 
The next statement estimates the growth of different pieces of $\sigma.$ Denote
\begin{equation}\label{eq:sigmadec}
	\sigma(t)=\sigma_1\lbrb{H_t}+\sigma_2\lbrb{t-H_t}=\sigma^1_1\lbrb{H^1_t}+\sigma^2_1\lbrb{H^2_t}+\sigma_2\lbrb{t-H_t},
\end{equation}
where $\sigma^i_1$ correspond to the processes localized to $A_i,i=1,2.$

Then we have the result.
\begin{coro}\label{cor:growth} Let $\alpha$ be such that the set $A$ satisfy the conditions above. Assume further that $\lim_{x\to\infty,x\notin A}\alpha(x)=\alpha_I$ and $\lim_{x\to-\infty,x\notin A}\alpha(x)=\alpha_J$. Then for all $\epsilon>0$ small enough
	\begin{equation}\label{eq:sigma1_11}
	\limi{t}\frac{\sigma_1\lbrb{H_t}}{t^{\frac{1+c_1}{2\alpha^*}+\epsilon}}=0;\quad \limi{t}\frac{\sigma^1_1\lbrb{H^1_t}}{t^{\frac{1+c_1}{2\alpha^*}-\epsilon}}=\infty,\,\,\text{a.s.}
	\end{equation}
	and with $\alpha_{\circ}=\min\curly{\alpha_I,\alpha_J}$
	\begin{equation}\label{eq:sigma1_12}
	\limi{t}\frac{\sigma_2\lbrb{t-H_t}}{t^\frac{1}{\alpha_{\circ}+\epsilon}}=0;\quad \limi{t}\frac{\sigma_2\lbrb{t-H_t}}{t^\frac{1}{\alpha_{\circ}-\epsilon}}=\infty,\,\,\text{a.s.}
	\end{equation}
\end{coro}
\begin{proof}
	Since $\alpha(x)=\alpha^*$ on $A$, we have that $\sigma_1=\sigma^{\alpha^*}$, that is stable with index $\alpha^*$ and thus for all $\epsilon$ small enough
	\begin{equation}\label{eq:stochBounds5}
	\begin{split}
	&\liminfi{t}\frac{\sigma^{\alpha^{*}}(H_t)}{\lbrb{H_t}^{\frac{1}{\alpha^*}-\epsilon}}=\infty;\text{ and }
	 \limsupi{t}\frac{\sigma^{\alpha^{*}}(H_t)}{\lbrb{H_t}^{\frac{1}{\alpha^*}+\epsilon}}=0.
	\end{split}
	\end{equation} 
	Then from \eqref{eq:inf1} and using $c_1\geq c_2$ we get the first relation of \eqref{eq:sigma1_11}. The second follows from the first relation of \eqref{eq:stochBounds5} and the fact that $c_1\geq c_2$ combined with \eqref{eq:growth}. Relation \eqref{eq:sigma1_12} is deducted precisely as in the proof of Corollary \ref{cor:stochBounds1} using the facts that $t-H_t\simi t$ since $1>c_1\geq c_2\geq 0$ and \eqref{eq:inf1}.
\end{proof}
We are now in a position to state the main result of this section.
\begin{te}\label{thm:infRegion}
	Let  $\alpha:\Rb\mapsto\lbrb{0,1}$, $\max_{x\in \Rb}\curly{\alpha(x)}<1$ and suppose that $\alpha^*=\min_{x\in\Rb}\alpha(x)>0$. Assume further that $A=\curly{x\in\Rb:\alpha(x)=\alpha^*<1}$ satisfies the growth Assumption (G) with $1>c_1\geq c_2\geq 0$, see \eqref{eq:asymp}. Also let $\mathpzc{l}\lbrb{\partial A}=0$ and $A=A_\beta=\curly{x\in\Rb:\alpha(x)<\alpha^*+\beta<1}$ for all $\beta>0$ small enough. Finally, set $\alpha_{\circ}=\min\curly{\alpha_I,\alpha_J}$, where we have $\lim_{x\to\infty,x\notin A}\alpha(x)=\alpha_I$ and $\lim_{x\to-\infty,x\notin A}\alpha(x)=\alpha_J$. Then,
	\begin{enumerate}
		\item if $\frac{2\alpha^*}{1+c_1}<\alpha_{\circ}$ we have that
		\begin{equation}\label{eq:localize}
		\begin{split}
		 &\limi{t}\frac{\int_{0}^{t}\ind{X(s)\in A}ds}{t}=1\text{ a.s.,}\\
		 &\limi{t}\Pbb{X(t)\in A}=1;
		\end{split}
		\end{equation}
		\item if $\frac{2\alpha^*}{1+c_1}>\alpha_{\circ}$ then for any $K>0$
		      \begin{equation}\label{eq:delocalize}
		      \begin{split}
		      & \limi{t}\frac{\int_{0}^{t}\ind{X(s)\in \lbbrbb{-K,K}^c\cap A^c}ds}{t}=1 \text{ a.s.}
		      \end{split}
		      \end{equation}
	\end{enumerate}
\end{te}
\begin{os}
	Inspection of the proof and the arguments shows that the result is also true provided $c_1=1$, that is $\limi{x}x^{-1}\mathpzc{l}\lbrb{A\cap\lbbrbb{-x,x}}=a\in\lbrb{0,\infty}$, where $A=\curly{x\in\Rb:\,\alpha(x)=\alpha^*}$. Then only \eqref{eq:localize} can hold.
\end{os}
\begin{proof}
The	first relation of \eqref{eq:localize} is established precisely as in the proof of Theorem \ref{thm:occupation} using the different growth for the occupation measure in this case. The	second relation of \eqref{eq:localize} is proved with the same method as in the proof of Theorem \ref{thm:hyp} noting that on $A$, $\sigma_1$ is stable subordinator of index $\alpha^*$ and the contradiction this would trigger thanks to \eqref{eq:Close11} of Proposition \ref{prop:Close} provided we assume that $\liminfi{t}\Pbb{X(t)\in A}<1$. Relation \eqref{eq:delocalize} is again as the proof of Theorem \ref{thm:occupation}.
\end{proof}

\vspace{1cm}

\begin{thebibliography}{16}
\providecommand{\natexlab}[1]{#1}
\providecommand{\url}[1]{\texttt{#1}}
\expandafter\ifx\csname urlstyle\endcsname\relax
  \providecommand{\doi}[1]{doi: #1}\else
  \providecommand{\doi}{doi: \begingroup \urlstyle{rm}\Url}\fi



\bibitem{abhn}
W. Arendt, C.J.K. Batty, M. Hieber and F. Neubrander.
\newblock Vector valued Laplace transform and Cauchy problem. Second Edition.
\newblock \emph{Birkh\"{a}user}, Berlin, 2010.


\bibitem{aurzada}
F. Aurzada, L. Doering and M. Savov.
\newblock{Small time Chung-type LIL for L\'evy processes.}
{\it Bernoulli}, 19(1): 115 -- 136, 2013.

\bibitem{fracCauchy} B. Baeumer and M.M. Meerschaert. Stochastic solutions for fractional Cauchy problems.  {\it Fractional Calculus and Applied Analysis} {4}: 481--500, 2001.




\bibitem{baemstra}
B. Baeumer and P. Straka. Fokker-Planck and Kolmogorov backward equations for continous time random walk limits.
  \emph{Proceedings of the American Mathematical Society}, 145: 399 -- 412, 2017.


\bibitem{bazhlekova} E. Bazhlekova. Subordination principle for fractional evolution equations. {\it Fract. Calc. Appl. Anal.} {3}({3}): {213--230}, 2000.

\bibitem{bazh15} E. Bazhlekova. Completely monotone functions and some classes of fractional evolution equations. {\it Integral Transforms and Special Functions}, 26(9): 737 -- 752, 2015.


\bibitem{costa} L. Beghin and C. Ricciuti. Time inhomogeneous fractional Poisson processes defined by the multistable subordinator. {\it Preprint}, available at arXiv:1608.02224.

\bibitem [Bertoin(1996)]{bertoinb}
J. Bertoin.
\newblock {L\'evy processes}.
\newblock \emph{Cambridge University Press}, Cambridge, 1996.

\bibitem [Bertoin(1997)]{bertoins}
J. Bertoin.
\newblock {Subordinators: examples and appications}.
\newblock \emph{Lectures on probability theory and statistics (Saint-Flour, 1997)}, 1 -- 91. \emph{Lectures Notes in Math.}, 1717, Springer, Berlin, 1999.

\bibitem {bingham}
N.H. Bingham, C.M. Goldie and J.F. Teugels.
\newblock {Regular variation}.
\newblock \emph{Cambridge University Press}, Cambridge, 1987.

%



  \bibitem [Schilling et al.(2010)]{schillinglevy}
B. B\"ottcher, R. Schilling and J. Wang.
\newblock {L\'evy Matters III. L\'evy-Type Processes: Construction, Approximation and Sample Path Properties.}
\newblock \emph{Springer}, 2013.

  \bibitem{campos}
D. Campos, S. Fedotov, and V. M\'endez.
\newblock {Anomalous reaction-transport processes: The dynamics beyond the law of mass action.}
\newblock \emph{Phys. Rev. E}, 77: 061130, 2008.

\bibitem{mirkorafasym}
R. Capitanelli and M. D'Ovidio.
\newblock {Asymptotics for time-changed diffusion}.
\newblock \emph{Theory of Probability and Mathematical Statistics,} special volume 95 in honor of Prof. N. Leonenko: 37 -- 54, 2017. 


\bibitem{mirkoraf}
R. Capitanelli and M. D'Ovidio.
\newblock {Fractional equations via convergence of forms}.
\newblock \emph{Preprint}, arXiv:1710.01147, 2017.






\bibitem{checgore} A.V. Chechkin, R. Gorenflo, and I.M. Sokolov. Fractional diffusion in inhomogeneous media  {\it J. Phys. A: Math. Gen.} 38: L679 -- L684, 2005.


\bibitem[Chen (2017)]{zqc} Z.-Q. Chen. Time fractional equations and probabilistic representation. \emph{Chaos, Solitons and Fractals}, 102: 168 -- 174, 2017.


\bibitem{cinlar1}
E. Cinlar.
\newblock {Markov additive processes. I.}
\newblock \emph{Z. Wahrscheinlichkeitstheorie verw. Geb.}, 24: 85 -- 93, 1972.

\bibitem{cinlar2}
E. Cinlar.
\newblock {Markov additive processes. II.}
\newblock \emph{Z. Wahrscheinlichkeitstheorie verw. Geb.}, 24: 95 -- 121, 1972.


\bibitem{cinlarlevy}
E. Cinlar.
\newblock {L\'evy system of Markov additive processes}.
\newblock \emph{Discussion Paper No. 63}, Northwestern University, 1973. 


\bibitem{cinlarsemi}
E. Cinlar.
\newblock {Markov additive processes and semi-regeneration}.
\newblock \emph{Discussion Paper No. 118}, Northwestern University, 1974.







\bibitem [Engel and Nagel(2000)]{engelnagel}
K.-J. Engel and R. Nagel.
\newblock {One-parameter semigroups for linear evolution equations}.
\newblock \emph{Springer Science \& Business Media}, 2000.

\bibitem{fedopre}
S. Fedotov.
\newblock Subdiffusion, chemotaxis, and anomalous aggregation.
\emph{Physical Review E}, 83: 021110, 2011.

\bibitem{fedofalco}
S. Fedotov and S. Falconer.
\newblock Subdiffusive master equation with space-dependent anomalous exponent and structural instability.
\emph{Physical Review E}, 85: 031132, 2012.






\bibitem[Garra et al.(2009)]{garra}
R. Garra, F. Polito and E. Orsingher.
\newblock{State-dependent fractional point processes.}
\newblock{\emph{J. Appl. Probab.}}, 52: 18 -- 36, 2015.


\bibitem[Georgiou et al.(2015)]{nicos}
N. Georgiou, I.Z. Kiss and E. Scalas.
\newblock Solvable non-Markovian dynamic network.
\emph{Physical Review E}, 92, 042801, 2015.


\bibitem[Gihman and Skorohod(1975)]{gihman}
I.I. Gihman and A.V. Skorohod.
\newblock The theory of stochastic processes II.
\emph{Springer-Verlag}, 1975.




\bibitem{hairer} 
M. Hairer, G. Iyer, L. Koralov, A. Novikov, and Z. Pajor-Gyulai.
A fractional kinetic process describing the intermediate time behaviour of cellular flows. \emph{The Annals of Probability}, 46(2): 897 -- 955, 2018.


\bibitem[Harlamov(2008)]{harlamov}
B.P. Harlamov.
\newblock{Continuous semi-Markov processes.}
\newblock{\emph{Applied Stochastic Methods Series.}}, ISTE, London; \emph{John Wiley \& Sons, Inc.}, Hoboken, NJ, 2008.


%
%
%



%

%
%
\bibitem{kaspi}
H. Kaspi and B. Maisonneuve.
\newblock{Regenerative systems on the real line}.
\newblock{\emph{Ann. Probab.}}, 16: 1306 -- 1332, 1988.



%
%


\bibitem[Kochubei(2011)]{kochu}
A.N. Kochubei.
\newblock{General fractional calculus, evolution equations and renewal processes}.
\newblock{\emph{Integral Equations and Operator Theory}}, 71: 583 -- 600, 2011.


\bibitem[Kolokoltsov(2009)]{KoloCTRW}
V.N. Kolokoltsov. Generalized Continuous-Time Random Walks, subordination by hitting times, and fractional dynamics. \emph{Theory Probab.
  Appl.} {53}:, 594--609, 2009.



\bibitem{kololast}
M. E. Hern\'andez-Hern\'andez, V.N. Kolokoltsov and L. Toniazzi.
\newblock{Generalised Fractional Evolution Equations of Caputo Type.}
\newblock{\emph{Chaos, Solitons \& Fractals,}} to appear.

\bibitem{barkai}
N. Korabel and E. Barkai.
\newblock{Paradoxes of Subdiffusive Infiltration in Disordered Systems.}
\newblock{\emph{Phys. Rev. Lett.,}} 104: 170603, 2010. 




%
%




\bibitem[Magdziarz and Schilling(2013)]{magda}
M. Magdziarz and R.L. Schilling.
\newblock{Asymptotic properties of Brownian motion delayed by inverse subordinators.}
\newblock{\emph{Proc. Amer. Math. Soc}}, 143: 4485 -- 4501, 2015. 
%
%
%


\bibitem [Meerschaert and Scheffler(2008)]{meertri}
M.M. Meerschaert and H.P. Scheffler.
\newblock {Triangular array limits for continuous time random walks}.
\newblock \emph{Stochastic Processes and their Applications}, 118(9): 1606 -- 1633, 2008.


\bibitem [Meerschaert et al.(2011)] {meerpoisson}
M.M. Meerschaert, E. Nane and P. Vellaisamy.
\newblock {The fractional Poisson process
and the inverse stable subordinator}.
\newblock \emph{Electronic Journal of Probability}, 16(59): 1600--1620, 2011.

\bibitem [Meerschaert et al.(2011)] {meerbounded}
M.M. Meerschaert, E. Nane and P. Vellaisamy.
\newblock {Fractional Cauchy problems on bounded domains}.
\newblock \emph{The Annals of Probability}, 37(3): 979 -- 1007, 2009.



\bibitem[Meerschaert and Sikorskii(2012)]{FCbook}
    M.M. Meerschaert and A.  Sikorskii. {\it Stochastic Models for Fractional Calculus}. De Gruyter,  Berlin, 2012.

\bibitem[Meerschaert and Straka(2014)]{meerstra}
M.M. Meerschaert and P. Straka.
\newblock {Semi-Markov approach to continuous time random walk limit processes}.
\newblock \emph{The Annals of Probability}, 42(4) : 1699 -- 1723, 2014.






\bibitem [Meerschaert and Toaldo(2015)] {meertoa}
M.M. Meerschaert and B. Toaldo.
\newblock {Relaxation patterns and semi-Markov dynamics}.
\newblock \emph{Preprint}.

\bibitem{Metzler}  R. Metzler and J. Klafter. The random walk's guide to anomalous diffusion: a fractional dynamics approach. \emph {Physics Reports}, 339: 1 -- 77, 2000.

\bibitem {niels2}
N. Jacob.
\newblock {Pseudo-differential operators and Markov processes. Vol II.}
\newblock \emph{Imperial College Press}, 2002.






\bibitem [Orsingher et al.(2016)]{orsrictoapota}
E. Orsingher, C. Ricciuti and B. Toaldo.
\newblock {Time-inhomogeneous jump processes and variable order operators}.
\newblock \emph{Potential Analysis}, 45(3): 435 -- 461, 2016.



\bibitem [Orsigher et al.(2017)]{orsrictoasemi}
E. Orsingher, C. Ricciuti and B. Toaldo.
\newblock {On semi-Markov processes and their Kolmogorov integro-differential equations}.
\newblock \emph{Journal of Functional Analysis}, 275(4): 830 -- 868, 2018.
%


%
%
%
%
\bibitem [Raberto et al.(2011)]{raberto}
M. Raberto, F. Rapallo and E. Scalas.
\newblock {Semi-Markov Graph Dynamics}.
\newblock \emph{Plos One}, 6(8): e23370, 2011.
%
\bibitem {ry}
D. Revuz and M. Yor.
\newblock {Continuous Martingales and Brownian Motion}.
\newblock \emph{Springer}, 1999.


\bibitem {rictoa}
C. Ricciuti and B. Toaldo.
\newblock {Semi-Markov models and motion in heterogeneous media}.
\newblock \emph{Journal of Statistical Physics}, 169(2): 340 - 361, 2017.


\bibitem{savov09} 
M. Savov.
\newblock{Small time two-sided LIL behavior for L\'evy processes at zero.}
\newblock{\it Probab. Theory Related Fields}, 144(1-2): 79 -- 98, 2009.



\bibitem{scalas}
E. Scalas. {Five years of continuous-time random walks in econophysics}.
  The Complex Networks of Economic Interactions, \emph{Lecture Notes in Economics and
  Mathematical Systems},  567: 3--16, Springer, Berlin, 2006.
  
  \bibitem [Schilling et al.(2010)]{librobern}
R.L. Schilling, R. Song and Z. Vondra\v{c}ek.
\newblock {Bernstein functions: theory and applications}.
\newblock \emph{Walter de Gruyter GmbH \& Company KG}, Vol 37 of De Gruyter Studies in Mathematics Series, 2010.

\bibitem{shushin}
A. I. Shushin. 
{Anomalous two-state model for anomalous diffusion}.
  \emph{Phys. Rev. E},  64, 051108, 2001.

 








\bibitem{stick}
B. A. Stickler and E. Schachinger.
\newblock{Continuous time anomalous diffusion in a composite medium.}
\newblock \emph{Phys. Rev. E - Stat. Nonlinear, Soft Matter Phys.,}, 84(2): 1 -- 9, 2011.


 

%
%


\bibitem{strakavariable}
P. Straka.
\newblock{Variable Order Fractional Fokker-Planck Equations derived from Continuous Time Random Walks.}
\newblock \emph{Physica A: Statistical Mechanics and its Applications}, 503: 451 -- 463, 2018.


\bibitem{strakafedo}
P. Straka and S. Fedotov.
\newblock{Transport equations for subdiffusion with nonlinear particle interaction.}
\newblock \emph{J. Theor. Biol.}, 366: 71 -- 83, 2015.



\bibitem{taleb}
N. Taleb.
\newblock {The Black Swan: The Impact of the Highly Improbable.}
\newblock \emph{Random House}, New York, 2007.



\bibitem[Toaldo(2014)]{toaldopota}
B. Toaldo.
\newblock {Convolution-type derivatives, hitting-times of subordinators and time-changed $C_0$-semigroups.}
\newblock \emph{Potential Analysis}, 42(1): 115--140, 2015.

\bibitem[Toaldo(2014)]{toaldodo}
B. Toaldo.
\newblock {L\'evy mixing related to distributed order calculus, subordinators and slow diffusions.}
\newblock \emph{Journal of Mathematical Analysis and Applications}, 430(2): 1009 -- 1036, 2015.


\bibitem{wong}
I.Y. Wong, M.L. Gardel, D.R. Reichman, E.R. Weeks, M.T. Valentine, A.R. Bausch, and D.A. Weitz.
\newblock {Anomalous Diffusion Probes Microstructure Dynamics of Entangled F-Actin Networks.}
\newblock \emph{Phys. Rev. Lett}, 92(17):178101, 2004.



%
\end{thebibliography}
\end{document}